\documentclass[10pt]{amsart}

\usepackage{amsmath, amssymb, amsthm, longtable, booktabs, float, graphicx, url, hyperref, comment}
\usepackage[all]{xy}

\DeclareMathOperator{\rank}{rank}

\newcommand{\F}{\mathbb{F}}
\newcommand{\Z}{\mathbb{Z}}

\newcommand{\C}{\mathbb{C}}
\newcommand{\PP}{\mathbb{P}}
\newcommand{\Q}{\mathbb{Q}}
\newcommand{\calA}{\mathcal{A}}

\newcommand{\calF}{\mathcal{F}}
\newcommand{\calG}{\mathcal{G}}
\newcommand{\w}{\omega}
\newcommand{\iso}{\cong}

\newcommand{\Sp}{\operatorname{sp}}
\newcommand{\lk}{\operatorname{lk}}
\newcommand{\slfrac}[2]{\left.#1\middle/#2\right.}

\makeatletter
\newtheorem*{rep@theorem}{\rep@title}
\newcommand{\newreptheorem}[2]{%
\newenvironment{rep#1}[1]{%
 \def\rep@title{#2 \ref{##1}}%
 \begin{rep@theorem}}%
 {\end{rep@theorem}}}
\makeatother

\newtheorem{theorem}{Theorem}[section]

\newtheorem{lemma}[theorem]{Lemma}
\newtheorem{question}[theorem]{Question}
\newtheorem{proposition}[theorem]{Proposition}
\newtheorem{definition}[theorem]{Definition}

\newtheorem{corollary}[theorem]{Corollary}
\newreptheorem{theorem}{Theorem}

\title{A Link Splitting Spectral Sequence in Khovanov Homology}

\author[Joshua Batson]{Joshua Batson}
\address{Department of Mathematics \\ Massachusetts Institute of Technology}
\email{joshua@math.mit.edu}

\author[Cotton Seed]{Cotton Seed}
\address{Department of Mathematics \\ Princeton University}
\email{cseed@math.princeton.edu}

\begin{document}
\begin{abstract} 
We construct a new spectral sequence beginning at the Khovanov
homology of a link and converging to the Khovanov homology of the
disjoint union of its components. The page at which the sequence
collapses gives a lower bound on the splitting number of the link, the
minimum number of times its components must be passed through one
another in order to completely separate them.  In addition, we build
on work of Kronheimer-Mrowka and Hedden-Ni to show that Khovanov
homology detects the unlink.

\end{abstract}

\maketitle

\tableofcontents

\section{Introduction}
\label{sec:intro}
	Quantum invariants of knots and $3$-manifolds have been an active area
of study since the discovery of the Jones polynomial in 1984 \cite{jones}.  Despite the
attention they have received, the geometric and topological content
of these invariants is not well understood. A promising turn was the discovery by Khovanov \cite{khovanov} of a homology theory
assigning bigraded groups to a link $L$ in $S^3$; the Jones polynomial
can be recovered as their graded Euler characteristic. 
Rasmussen used a spectral sequence from Khovanov homology to 
Lee homology to construct the
$s$-invariant, a lower bound on the slice genus of a knot \cite{ras}.

In this paper we use Khovanov homology to bound a simple topological
invariant of a link, its splitting number. Roughly, the splitting
number of a link is the minimum number of times its components must be
passed through one another in order to completely separate them.  For
example, the two-component link in Figure~\ref{boundexfig} can be
split into an unknot and a trefoil by changing three crossings. We
show that this is best possible.

Our bound comes from a new spectral sequence beginning at the Khovanov
homology of a link and converging to the Khovanov homology of the disjoint union of its
components. 

\begin{theorem}\label{SseqThm}
Let $L$ be a link and $R$ a ring. Choose weights $w_c \in R$ for each component $c$ of $L$. Then
there is a spectral sequence with pages $E_k(L,w)$, and
\[E_1(L,w) \cong Kh(L;R).\]
If the difference $w_c-w_d$ is invertible in $R$ 
for each pair of components $c$ and $d$ with distinct weights,
 then the spectral sequence converges to
\[Kh \left (\coprod_{r\in R} L^{(r)};R \right),\]
where $L^{(r)}$ denotes the sub-link of $L$ consisting of those components
with weight $r$.
\end{theorem}

Each choice of weights for a link $L$ gives a lower bound on the splitting number.

\begin{theorem}\label{SplittingNumThm}
  Let $L$ be a link and let $w_c\in R$ be a set of component weights such
  that $w_c - w_d$ is invertible for each pair of components $c$ and $d$. Let $b(L,
  w)$ be largest $k$ such that $E_k(L, w) \neq E_\infty(L, w)$.
  Then $b(L, w)\le \Sp(L)$.
\end{theorem}

We also use the spectral sequence to show that the Poincar\'e polynomial of Khovanov homology detects the unlink. In contrast, there are infinite families of links with the same Jones polynomial as the unlink \cite{ekt, thistle}.

\begin{theorem}\label{ThmUnlinkDetection} 
  Let $L$ be an $m$-component link, and $U^m$ the $m$-component unlink.  If
\[ \rank Kh^{i,j}(L;\F_2) = \rank Kh^{i,j}(U^m;\F_2) \]
  for all $i, j$, then $L$ is the unlink.
\end{theorem}

The proof of Theorem~\ref{ThmUnlinkDetection} depends on two earlier spectral sequences that relate Khovanov homology to more manifestly geometric invariants coming from Floer homology. The first, constructed by Ozsv\'ath and
Szab\'o, begins at the Khovanov
homology of a link and converges to the Heegaard Floer homology of
its branched double cover \cite{ozsszo}. The second, constructed by Kronheimer and Mrowka, 
begins at the Khovanov homology of a knot and converges to its instanton knot Floer homology \cite{km}. The latter group is nontrivial for
nontrivial knots, so:
\begin{theorem}[Kronheimer-Mrowka]
Let $K$ be a knot, and $U$ the unknot. If 
\[ \rank Kh(K) = \rank Kh(U), \]
then $K$ is the unknot.
\end{theorem}

\begin{figure}[H]
  \begin{center}
    \includegraphics[scale=0.3]{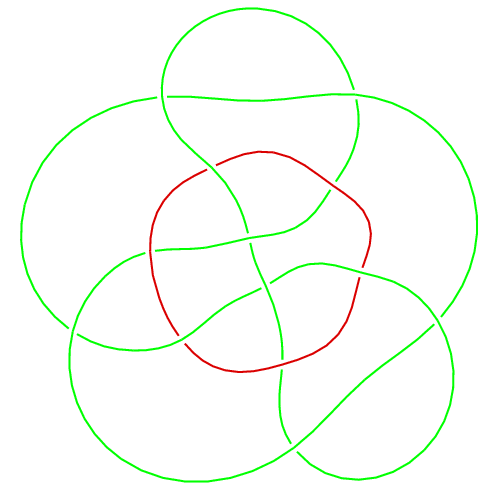}
  \end{center}
  \caption{The link ${}^2n^{13}_{8862}$ has splitting number $3$.\label{boundexfig}}
\end{figure}

The Khovanov homology groups contain more information than their ranks alone---there
is a natural action of the algebra
\[ A_m = \F_2[X_1,\dots,X_m]/(X_1^2, X_2^2, \dotsc, X_m^2)  \]
on the homology of an $m$-component link.
Hedden and Ni \cite{heddenni} showed that the entire spectral sequence of  Ozsv\'ath and
Szab\'o admits a compatible $A_m$ action. They then used of Floer homology
to detect $S^1 \times S^2$ summands in the branched double cover of the link, and showed:
\begin{theorem}[Hedden-Ni]\label{HeddenNiThm}
Let $L$ be an $m$-component link, and $U^m$ the $m$-component unlink. If there
is an isomorphism of $A_m$ modules
\[ Kh(L;\F_2) \cong Kh(U^m;\F_2), \]
then $L$ is the unlink.
\end{theorem}

To prove Theorem ~\ref{ThmUnlinkDetection}, we apply our spectral sequence with component weights in a suitably large finite field $\F$ of characteristic $2$. We lift the $A_m$-module structure from the abutment of our spectral sequence, which turns out to be isomorphic to $Kh(U^m;\F)$, to the first page, $Kh(L;\F)$, and then to $Kh(L;\F_2)$, where we apply Theorem~\ref{HeddenNiThm}.

\noindent {\bf Deformations.} We construct our spectral sequence by giving a filtration-preserving deformation of the
differential in Khovanov's cube of resolutions complex. That deformation may have analogues in other flavors of Khovanov homology. 

In \cite{seidelsmith}, Seidel and Smith associate
to a braid $\beta$ a pair of Lagrangians in an exact symplectic manifold $\mathcal{Y}_{n,t_0}$.
The Lagrangian Floer homology of that pair is conjectured to be isomorphic to the Khovanov homology
of the closure of $\beta$ (Abouzaid and Smith have announced a proof \cite{sympkheq}). In \S\ref{sec:symp}, we discuss how a perturbation of the symplectic form away
from exactness could give a result in symplectic Khovanov homology similar to the rank inequality implied by Theorem~\ref{ThmUnlinkDetection}:
\begin{corollary}\label{CorRankBounds}
Let $\F$ be any field, and let $L$ be a link with components $K_1,\dots,K_m$. Then
\[ \rank Kh^*(L;\F) \geq \rank \otimes_{c=1}^m Kh^*(K_c;\F). \]
\end{corollary}

Cautis and Kamnitzer have given a description of Khovanov homology in terms
of derived categories of coherent sheaves \cite{cautkam}, which is mirror
to symplectic Khovanov homology. It would be interesting
to see how our deformation manifests in their framework.

There is also a Khovanov homology for tangles \cite{khovanov2}, 
which associates to a tangle $T$ a complex $Kh(T)$ of bimodules over
certain rings. In a forthcoming article, we will construct a deformation of
$Kh(T)$ to a curved complex, analogous to the curved Fukaya category
of a non-exact symplectic manifold \cite{forgetfultang}. When the
tangle is a closed link, this will recover the spectral sequence constructed here. 
We suspect that a similar deformation exists in the
Khovanov-Rozansky theory of matrix factorizations for
$\mathfrak{sl}_n$ link homology \cite{khovanovrozansky}.

\noindent {\bf Outline.} In \S\ref{sec:ourconst},
we recall the construction of the Khovanov complex, define a filtered
chain complex $C(D, w)$ which induces the spectral sequence $E_k$, 
 and compute the $E_1$ and $E_\infty$ pages.  
 In \S\ref{sec:inv}, we verify that our spectral sequence is
 independent of the choice of link diagram by checking invariance under the Reidemeister moves.
 In \S\ref{sec:ssreview}, we review how
endomorphisms of a filtered complex act on the associated spectral
sequence and discuss the effect of changing the filtration.
In \S\ref{sec:splitnum}, we prove Theorem \ref{SplittingNumThm} on the 
splitting number.
In \S\ref{sec:unlinkdetect}, we give the proof of unlink detection. 
In \S\ref{sec:symp} we discuss the relationship with symplectic topology.
Finally, in \S\ref{sec:ex}, we discuss
computations illustrating the strength of this spectral sequence and
the splitting number bound.

\noindent {\bf Acknowledgements.}  The authors would like to thank Yi Ni for asking if such a spectral sequence might exist, Paul Seidel for suggesting the deformation of symplectic Khovanov homology, and Eli Grigsby for suggesting unlink detection as a potential application and introducing us to \cite{heddenni}. We also wish to thank Peter Ozsv\'ath, Sucharit Sarkar and Zoltan Szab\'o for helpful conversations. The link diagrams in this paper were rendered by Knotilus \cite{knotilus}. The work of the first author was partially supported  by NSF grants DMS-0603940, DMS-1006006 and EMSW21-RTG.

\section{Our construction}
\label{sec:ourconst}
	Khovanov's construction begins with a diagram $D$ for a link $L$. He builds a cube of resolutions for $D$ and applies a $(1+ 1)$-dimensional TQFT $\calA$ to produce a cube-graded complex. A sprinkling of signs yields a chain complex $(C(D), d_0)$, with homology $Kh(L)$. We will give another differential $d$ on the same chain complex, but first we must set some notation.

	\subsection{A review of Khovanov homology}
		(Following \cite{khovanov} and \cite{bncat}.)

A crossing in a link diagram can be resolved in two ways, called the 0-resolution and 1-resolution in Figure~\ref{resolutionfig}. 
A (complete) resolution of $D$ is a choice of resolution at each crossing. Order the crossings
of $D$ from $1$ to $n$ so we can index complete resolutions by vertices in the hypercube $\{0, 1\}^n$.
An edge in the cube connects a pair of resolutions $(I, J)$, where $J$ is obtained from $I$ by changing the $i^\text{th}$ digit from $0$ to $1$. A complete resolution $I$ yields a finite collection of circles in the plane, which we may also call $I$. An edge $(I,J)$ yields a cobordism from $I$ to $J$, given by the natural saddle cobordism from the $0$- to the $1$-resolution in a neighborhood of the changing crossing and the product cobordism elsewhere.

\begin{figure}
  \begin{center}
    \includegraphics{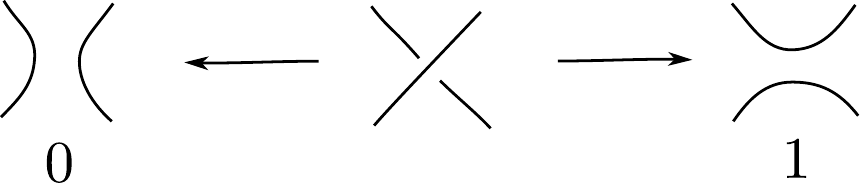}
  \end{center}
  \caption{The $0$ and $1$ resolutions associated to a
    crossing.\label{resolutionfig}}
\end{figure}

A $(1 + 1)$-dimensional TQFT is determined by commutative Frobenius algebra \cite{ko}. We fix a ring of coefficients $R$, and let $\calA$ be the TQFT associated to the Frobenius algebra $V = H^*(S^2;R) = R[x]/(x^2)$. The diagonal map $i: S^2 \hookrightarrow S^2 \times S^2$ induces the multiplication $i^*:H^*(S^2 \times S^2)  \rightarrow H^*(S^2)$. The comultiplication comes from Poincar\'{e} duality, $PD\circ i_* \circ PD: H^*(S^2) \rightarrow H^*(S^2 \times S^2)$. More explicitly, the multiplication $m : V\otimes V\to V$ is given by
\begin{align*}
  m(1\otimes 1) &= 1 & m(x\otimes 1) &= x \\
  m(1\otimes x) &= x & m(x\otimes x) &= 0,
\end{align*}
and the comultiplication $\Delta : V\to V\otimes V$ is given by
\begin{align*}
  \Delta(1) &= 1 & \Delta(x) &= 1\otimes x + x\otimes 1.
\end{align*}

The TQFT $\calA$ associates to a circle the $R$-module $V$ and
 takes disjoint unions to tensor products. The pair of pants cobordism that merges
two circles into one induces the multiplication map $m$, and the pair of pants cobordism that splits
one circle into two induces the comultiplication map $\Delta$.

Let $S = (x_1, \dotsc, x_p)$ be a collection of circles.  To simplify
notation, we note that
\begin{align*}
  \calA(S) &= \bigotimes_{i=1}^p V \\
  &= R[x_1, \dotsc, x_t]/(x_1^2, \dotsc, x_p^2).
\end{align*}
We will write elements of $V(S)$ as (commutative) products of
the circles $x_i$ rather elements of the tensor product.  Such a
product of circles is called a monomial of $S$.

Applying the TQFT $\calA$ to the cube of resolutions, we obtain a
cube-graded complex of $R$-modules.  For each resolution $I$, we have
an $R$-module $\calA(I)$, and for each edge $(I, J)$, we have a
homomorphism $\calA(I, J) : \calA(I)\to \calA(J)$.  Khovanov's complex
is obtained by collapsing the cube-graded complex.  We set
\[ C(D) = \bigoplus_{\text{resolutions $I$}} V(I). \]
The differential $d_0 : C(D)\to C(D)$ is given by
\[ d_0 = \sum_{\text{edges $(I, J)$}} (-1)^{n(I, J)} \calA(I, J), \]
where, if $(I, J)$ differ at $i$,
\[ n(I, J) = \#\{ I(k) = 1 \,|\, 1\le k < i \}. \]

We define four related gradings on $C(D)$ as follows. Let $x \in V(I)$.
The homological or $h$ grading is given by
\[ h(x) = |I| - n_-(D), \]
where $|I|$ number of 1 digits in $I$ and $n_-(D)$ is
the number of negative crossings in $D$. Monomials in $V^{\otimes p}$ have a natural degree induced by 
\[ \deg(1) = 0 \text{ and } \deg(x_i) = 2. \]
The internal or $\ell$ grading is given by 
\[ \ell(x) = \deg(x) - p(I) - \mbox{writhe}(D), \]
where  $p(I)$ is the number of circles in the resolution $I$.
The quantum or $q$ grading is given by
\begin{align*}
  q(x) &= h(x) - \ell(x) \\
\end{align*}
Finally, we define the $g$ grading, a normalization of the $q$ grading, by
\[ g(x) = \frac{q(x) - m}{2}, \]
where $m$ is the number of components of $L$. (It turns out that $g$ is always an integer \cite[\S 6.1]{khovanov}.)  The $g$ grading will induce the filtration on $C(D)$ in the definition of our spectral sequence.

Khovanov's differential $d_0$ increases both $h$ and $\ell$ by $1$, so it preserves $q$ and $g$. Khovanov homology is
\[ Kh(L) = H^*(C(D), d_0), \]
and has a bigrading given by $(h,q)$.

A choice of marked point on the diagram $D$ induces an endomorphism of
Khovanov homology \cite{khpatterns}:  Let $p$ be a marked point on $L$
away from the double points.  For a resolution $I$, let $x_p = x_p(I)$
denote the circle of $I$ meeting $p$.  Define a map $X_p : C(D)\to
C(D)$ by 
\[ X_p(x) = x_px \]
for $x \in V(I)$.
The map $X_p$ is a chain map and shifts the $(h, q)$ bigrading by $(0,
-2)$.  The map induced on homology, which we also call $X_p$, depends
only on the marked component, and not on the choice of marked point.

	\subsection{Our construction}
		
We begin by describing our construction in the case of a two-component
link $L$ with coefficients in $\F_2$. Khovanov's construction 
assigns a bigraded chain complex $(C(D),d_0)$ to a planar diagram $D$
for $L$.  We will give an endomorphism $d_1$ of $C(D)$ such that
\begin{enumerate}
\item $d:=d_0+d_1$ is a differential, which increases the $\ell$-grading by 1.
\item $d_1$ lowers the $g$-grading by 1, making $(C(D),d)$ a $g$-filtered complex.
\item If $i$ is a crossing in $D$ involving strands from different components of $L$ (a \emph{mixed} crossing), and $D'$ is the diagram for a link $L'$ produced by changing over-strand to under-strand at $i$, then $(C(D),d)$ and $(C(D'),d')$ are isomorphic chain complexes (with different $g$-filtrations).
\end{enumerate}

The new endomorphism is
\[ d_1 = \sum_{\text{mixed edges $(I, J)$}} \calA(J, I), \label{twocompdiff}\]
where an edge in the cube of resolutions is mixed if the $I$ and $J$ differ at a mixed crossing, and $(J, I)$ denotes the cobordism $(I, J)$ viewed backwards as a cobordism from $J$ to $I$.
The total differential
\[ d = \sum_{\text{non-mixed edges $(I, J)$}} \calA(I, J) + \sum_{\text{mixed edges $(I, J)$}} \calA(I, J) + \calA(J, I)  \]
is manifestly unchanged if we swap a mixed crossing. The square $d^2$ can have a component from $V(I)$ to $V(J)$ only when $I$ and $J$ differ at $2$ crossings or when $I = J$. The former vanish because they come in commuting squares (all maps are induced by cobordisms, and those commute due to the TQFT). The latter will vanish too, essentially because each circle in a complete resolution must have an even number of mixed crossings.

To define the endomorphism $d_1$ when there are more than two components, or over bigger rings, we need some additional data. First, we must weight each component by an element of the coefficient ring $R$: component $c$ has weight $w_c$. Then we must construct a sign assignment so that $d^2$ will be zero, not just even. As usual, different choices of sign assignment will produce isomorphic complexes.

We now define a sign assignment.  The shadow of the diagram $D$ in the plane
gives a CW decomposition $X$ of $S^2$: the $0$-cells are the double
points of the diagram, the $1$-cells are the the $2n$ edges between
the crossings (oriented by the orientation of the link), and the
$2$-cells are the remaining regions (with the natural orientation induced from $S^2$).  For a
$1$-cell $e$, let $e(0)$ denote the initial vertex and $e(1)$ denote
the final vertex so that $\partial e = e(1) - e(0)$.

Let
\[ h(e, i) = \begin{cases}
  1 & \text{$e$ is an upper strand at $e(i)$} \\
  -1 & \text{$e$ is a lower strand at $e(i)$},
  \end{cases}
\]
where $i\in \{0, 1\}$.  There is a natural $1$-cochain $\beta : X^1\to
\Z/2$, where $\Z/2 = \{1, -1\}$ is written multiplicatively, given by
\[ \beta(e) = \begin{cases}
  -1 & h(e, 0) = h(e, 1) \\
  1 & \text{otherwise}.
\end{cases}
\]

A {\em sign assignment} is a $0$-cochain $s : X^0\to \Z/2$ such that
\begin{align}\label{SignEq}
  s(e(0))s(e(1)) = \beta(e),
\end{align}
for all $1$-cells $e$.  This is equivalent to $\delta s = \beta$. Note that if $D$ is an alternating diagram, then $s \equiv 1$ is a legal sign assignment.
In the definition of $d_1$, we will use $s$ to sign the weight of the top strand at each crossing; the bottom strand will get the opposite sign.  The condition $\delta s =
\beta$ means that at adjacent crossings, connected by a strand in component $c$ of the link, the weight $w_c$ will appear with opposite signs in the contributions from each.

We now define the endomorphism $d_1$ of $C(D)$ as
\[ d_1 = \sum_{\text{edges (I, J)}} (-1)^{n(I, J)} s(i) (w^i_\text{over} - w^i_\text{under}) \calA(J, I), \]
where $I$ and $J$ differ at the $i^\text{th}$ crossing, and $w^i_\text{over}$
and $w^i_\text{under}$ are the weights of the over- and under-strands at
the $i^\text{th}$ crossing. Only the differences of weights appear, so shifting all
the weights by some $r\in R$, leaves the complex invariant.

In particular, the complex for a two-component link is determined by the choice of a single value $w_1-w_2 \in R$. If that difference is $1\in \F_2$, then this definition of $d_1$ reduces to \eqref{twocompdiff}.

The complex $(C(D),d = d_0 + d_1)$ now satisfies properties (1) and (2) from the beginning of this section. Both $d_0$ and $d_1$ increase the (internal) $l$-grading by $1$.  The differential $d_0$ preserves the $g$ grading and $d_1$ decreases the $g$ grading by $1$.  So we have a $g$-filtration on $(C(D), d)$ given by
\[ \calF^pC(D) := \{ x \,|\, x\in C(D), g(x)\le p \}. \]
Moreover, the spectral sequence associated to this filtration has
$E_1$ page given by $H^*(C(D),d_0) \cong Kh(L)$.

We now show it is always possible to choose a sign assignment.

\begin{figure}
  \begin{center}
    \includegraphics{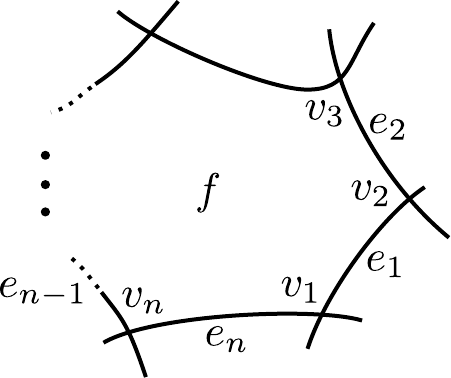}
  \end{center}
  \caption{\label{betacyclefig}}
\end{figure}

\begin{proposition}
  Let $D$ be a connected diagram.  There are precisely two sign
  assignments $s_1$ and $s_2$ for $D$, and $s_1 = -s_2$.
\end{proposition}
\begin{proof}
By (\ref{SignEq}), a choice of sign at one crossing determines the sign assignment for a connected diagram, if one exists.  Existence is a simple cohomological argument. Since a sign assigment is just a cochain $s\in C^0(S^2)$ with $\delta s = \beta $, such an $s$ exists if and only if $\beta \in C^1(S^2)$ is exact, and is unique up to multiplication by an element of $H^2(S^2)=\{\pm 1\}$.  Since $H^1(S^2) = 0$, $\beta$ is exact if and only if it is closed.

  We now show that $\beta$ is closed. Let $f$ be a $2$-cell with the incident $0$- and $1$-cells numbered
  and counterclockwise $v_1, \dotsc, v_n$ and $e_1, \dotsc, e_n$,
  respectively; see Figure~\ref{betacyclefig}.  We have $v_i = e_i(0) =
  e_{i-1}(1)$, where we set $v_{n+1} = v_1$ and $e_{n+1} = e_1$.  Let $t_1 = 1$ and 
  \[ t_i = \beta(e_1)\beta(e_2)\dotsm \beta(e_{i-1}) \qquad (1 < i\le n). \]
  We claim
  \[ t_i = \begin{cases}
    1 & h(e_1, 0) = h(e_i, 0) \\
    -1 & \text{otherwise.}
    \end{cases}
  \]
  It is trivial for $i = 1$.  The inductive step follows by $4$-way
  case analysis on $e_{i-1}$.  Since $h(e_{n+1}, 1) = h(e_1, 0)$, we
  have that
  \[ \beta(e_1)\beta(e_2)\dotsm \beta(e_n) = 1. \]
  Finally,
  \begin{align*}
    (\delta\beta)(f) &= \beta(e_1)\beta(e_2)\dotsm \beta(e_n) \\
    &= 1.\qedhere
  \end{align*}
\end{proof}

For a split diagram, sign assignments can be chosen on each connected
component independently.

Property (3) does not hold on the nose. If $D$ and $D'$ are related by changing a crossing, then the associated differentials $d$ and $d'$ are not identical---they differ by elements of $R$.  We will investigate this in Subsection~\ref{subsec:totalhom} after verifying that our new differential squares to zero and showing the filtered chain homotopy type of the $(C(D), d)$ does not depend on the choice of sign assignment.

		\begin{figure}
  \begin{center}
    \includegraphics{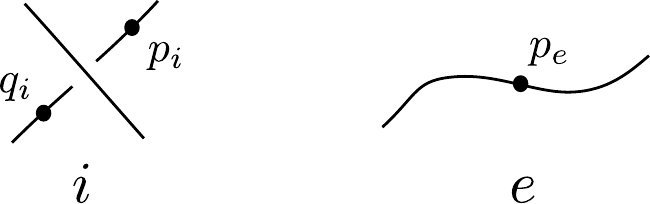}
  \end{center}
  \caption{We choose marked points $p_i$ and $q_i$ on the understrands
    at each crossing $i$ (left) and a marked point $p_e$ on each edge
    $e$ (right).\label{marksfig}}
\end{figure}

\begin{proposition}\label{complexprop}
  We have that $d^2 = 0$.
\end{proposition}
\begin{proof}
  Fix a resolution $I$ and let $x \in V(I)$.  The terms of $d^2(x)$
  lie in $V(K)$ where $K$ differs from $I$ in exactly two positions or
  $K = I$ itself.  We study these two cases.
  
  \textbf{Case 1}.  Let $K$ be a resolution that differs from $I$ in
  exactly two positions $i, j$ with $i < j$.  Let $J$ differ from $I$
  at $i$, and $J'$ differ from $I$ at $j$.  Then $I, J, J'$ and $K$
  are the four vertices of a face of the hypercube of resolutions.  By
  functoriality of $\calA$, we have that $\calA(J,K) \calA(I,J) =
  \calA(J',K)\calA(J',K)$.  The endomorphism $d_1$ uses the usual
  Khovanov sign assignments, so the two paths around the face have
  different signs.  Namely, we have that $n(I,J) = n(J',K)$ and
  $n(J,K) = - n(I,J')$.  The weights on the cobordism maps in
  $d_0$ and $d_1$ depend only on which crossing is changed, not
  the edge of the cube.  Denote the weights involved by $c(k)$,
  where
  \[ c(k) = \begin{cases}
    1 & I(k) = 0 \\
    s(k) (w^k_\text{over} - w^k_\text{under}) & I(k) = 1.
  \end{cases} \]
   The terms of
  $d^2(x)$ in $V(K)$ are
  \begin{align*}
    &\quad c(i)c(j) ((-1)^{n(I,J) + n(J,K)} \calA(J,K) \calA(I,J)(x) \\
    &\qquad \qquad \qquad  + (-1)^{n(I,J') + n(J',K)} \calA(J',K)\calA(I,J')(x) \\
    &= c(i)c(j) (-1)^{n(I,J) + n(J,K)} (\calA(J,K) \calA(I,J)(x) - \calA(J',K) \calA(I,J')(x)) \\
    &= 0.
  \end{align*}
  
  \textbf{Case 2}.  The terms of $d^2(x)$ in $V(I)$ are
  \[ \sum_{i=1}^n s(i)(w_\text{over} - w_\text{under}) \calA(J_i,I)\calA(I,J_i)(x), \]
  where $J_i$ is the resolution which differs from $I$ solely at the
  position $i$.  We choose
  marked points on the under-strands at each crossing and each edge,
  see Figure~\ref{marksfig}.  Straightforward computation shows that
  $\calA(J_i,I)\calA(I,J_i) = X_{p_i} + X_{q_i}$.  There is one under-strand
  edge and one over-strand edge in the circle containing $p_i$, and
  similarly for $q_i$, so we can rewrite the above sum as
  \begin{align*}
    &\sum_{i=1}^n s(i)(w_\text{over} - w_\text{under})(X_{p_i} + X_{q_i}) \\
    &= \sum_{e\in X^1} s(e(0))h(e,0)w_eX_{p_e} + s(e(1))h(e,1) w_eX_{p_e} \\
    &= \sum_{e\in X^1} (s(e(0))h(e,0) + s(e(1))h(e,1))w_eX_{p_e} \\
    &= 0,
  \end{align*}
  where $w_e$ denotes the weight of the component containing the edge
  $e$ and the final equality follows from the definition of a sign
  assignment.
\end{proof}

	\subsection{Change of sign assignment}
		While finding a sign assignment $s$ is crucial for defining the
complex over rings where $2 \neq 0$, different choices produce
isomorphic complexes. Indeed, consider a connected diagram $D$, weight
$w$, and sign assignment $s$ producing the complex $(C(D), d = d_0 +
d_1)$. Then taking the other sign assignment, $-s$, yields the
differential $d' = d_0 - d_1$ on the same group of chains
$C(D)$. Since $d_0$ fixes $g$-grading, and $d_1$ lowers it by $1$, the
endomorphism
\begin{align*}
\phi: C(D) &\rightarrow C(D) \\
x &\mapsto (-1)^{g(x)} x
\end{align*}
has the property that $d \phi = \phi d'$. That is, $\phi$ is an
invertible chain map between $(C(D),d)$ and $(C(D),d')$.

Next, consider the case when $D$ is possibly split and $s$ and $s'$
are two sign assignments.  Then, since $\calA$ is a monoidal functor,
the complexes $(C(D), d)$ and $(C(D), d)$ each decomposes into a
tensor product of complexes indexed over the components of $D$.  The
above analysis gives a chain equivalence $\phi$ for each component,
and their tensor product gives an invertible chain map between $(C(D),
d)$ and $(C(D), d')$.

Henceforth, we will often suppress the choice of a sign assignment,
writing $C(D,w)$ to indicate one of the two possible complexes.

	\subsection{Total homology}
        \label{subsec:totalhom}
		
We now show that changing a crossing doesn't affect the total homology of $(C(D),d)$, so long as the relevant weight $w_\text{over}-w_\text{under}$ is invertible. Of course, changing the crossing does {\em not} preserve the $g$-filtration on $C$.

\begin{figure}
  \begin{center}
    \includegraphics{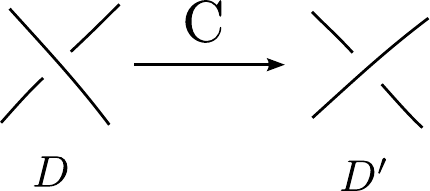}
  \end{center}
  \caption{The crossing change move C.\label{Cfig}}
\end{figure}

\begin{proposition}\label{CrossingChangeProp}
  Let $D$ and $D'$ be diagrams for links $L$ and $L'$ related by changing a crossing $i$ between components $c$ and $d$. Let $w$ be a weighting for $L$, and write $w'$ for the induced weighting on $L'$. Then if $w_c - w_d$ is invertible in $R$, the complexes $C(D, w)$ and $C(D', w')$ are isomorphic as relatively $\ell$-graded chain complexes.
\end{proposition}
\begin{proof}
  Let $s$ be a sign assignment for $D$.  A sign assigment $s'$  for $D'$ is given by $s'(j) =
  s(j)$ for $j\neq i$ and $s'(i) = -s(i)$.  Let $(C,d)$ be the complex $C(D,w,s)$, and let $(C',d')$ be
  the complex $C(D',w,s')$.  Let $C_0$ be the summand of $C$
  consisting of complete resolutions which include the $0$ resolution
  at crossing $i$, and let $C_1$, $C'_0$ and $C'_1$ be defined
  analgously.  Note that $C_0$ and $C'_1$ are identical as
  relatively $\ell$-graded complexes; similarly for $C_1$ and $C'_0$. 
  (The writhes of the diagrams differ by $2$, which will contribute
  a global shift between their $\ell$-gradings.)
  
  The components of $L$ are preserved by the crossing change,
  except $(w'^i_\text{over} - w'^i_\text{under}) = -(w^i_\text{over} -
  w^i_\text{under})$, since the upper and lower strands
  have been exchanged.  This means that
  \[ s(i)(w^i_\text{over} - w^i_\text{under}) = s'(i)(w'^i_\text{over} - w'^i_\text{under}). \]
  
 Before giving the chain map $f : C\to C'$, we must first introduce some notation.  Let $I$ be a resolution of $D$.  We write $I'$ to
  denote the same element of $\{0, 1\}^n$ interpreted as a resolution
  of $D'$.  We write $I_i$ for the resolution of $D$ that differs with
  $I$ solely at crossing $i$.  Note that $I$ and $I'_i$ are
  canonically isomorphic resolutions.  Let $J$ denote a resolution of
  $D$ that differs from $I$ at some crossing $j\neq i$.  Finally, let
\[ a(I, i) = \#\{ I(k) = 1 \,|\, i < k\le n \} \]
  be the number of one digits in $I$ above $i$.
  
  We define the map $f : C\to C'$ as follows.  
  	
\[ x \mapsto \begin{cases}
   & (-1)^{a(I, i)}x \text{ if } x \in V(I) \subset C_0 \\
   & (-1)^{a(I_1, i)} s(i)(w^i_\text{over} - w^i_\text{under}) \text{ if } x \in V(J) \subset C_1,
  \end{cases}
\]  
  To verify $f$ is a chain map, we use two easily verifiable facts about the signs: 
  \[(-1)^{a(I, i)} = (-1)^{a(I_i, i)} \] 
  and
\[ (-1)^{n(I, J)} (-1)^{a(J, i)} = (-1)^{n(I'_i, J'_i)} (-1)^{a(I, i)}. \]
  
  Consider $x\in V(I)\subset C_0$.  The image of
  $x$ under $fd$ or $d'f$ has components in $V(I')$ and $V(J'_i)$, for the resolutions $J$ differing from $I$ at one crossing.
  
  First, consider the $V(I')$-component of the image .  We have
  \begin{align*}
    fd(x) &= f((-1)^{n(I,I_i)} \calA(I, I_i)(x)) \\
    &= (-1)^{a(I_i, i)} (-1)^{n(I,I_i)} s(i)(w^i_\text{over} - w^i_\text{under}) \calA(I,I_i)(x) \\
    &= (-1)^{a(I, i)} (-1)^{n(I'_i,I')} s'(i)(w'^i_\text{over} - w'^i_\text{under}) \calA(I'_i, I')(x) \\
    &= d'((-1)^{a(I, i)} x) \\
    &= d'f(x).
  \end{align*}
  
  Next, consider the image in $V(J'_i)$ for some $J$ which differs from $I$
  at crossing $j$.  Let
\[ c(j) = \begin{cases}
  1 & I(j) = 0 \\
  s(i)(w^j_\text{over} - w^j_\text{under}) & I(j) = 1
\end{cases}
\]
  denote the coefficient of $\calA(I, J)$ in $d$.  It is the same as
  the coeffient of $\calA(I'_i, J'_i)$ in $d'$.  We have
\begin{align*}
  fd(x) &= f((-1)^{n(I,J)} c(j) \calA(I, J)(x)) \\
  &= (-1)^{a(J,i)} (-1)^{n(I,J)} c(j) \calA(I, J) \\
  &= (-1)^{a(I, i)} (-1)^{n(I'_i,J'_i)} c(j) \calA(I'_i, J'_i) \\
  &= d((-1)^{a(I, i)} x) \\
  &= df(x).
\end{align*}
  
  A similar analysis shows that $fd(x) = d'f(x)$ for $x\in
  V(I_1)\subset C_1$.
  
  Let $f': C'\to C$ be the chain map produced by reversing the roles of $D$ and $D'$.  The composition  
  \[ff'
  = f'f = s(i)(w^i_\text{over} - w^i_\text{under})\]
  is an isomorphism if $(w^i_\text{over} - w^i_\text{under})$ is invertible for all $i$. In that case,
  $f$ is a chain homotopy equivalence.
\end{proof}

\section{Reidemeister invariance}
\label{sec:inv}
	The proof that the filtered chain homotopy type of $C(D,w,s)$ is invariant under the Reidemeister moves parallels  the standard proof that the Khovanov chain complex is invariant. We divide the complex into the summands corresponding to the $2,4,$ or $8$ ways of resolving the crossings involved in the move, and cancel isomorphic summands along components of the differential. This is complicated slightly by the $d_{1}$ terms which prevent the natural summands of $C$ from being subcomplexes; the post-cancellation differential is not merely a restriction of the original one. The new differential is provided by the following standard cancellation lemma.

	\begin{lemma}
Let $(C,d)$ be a chain complex. Suppose that $C$, viewed as an $R$-module, splits as a direct sum $V \oplus W \oplus C'$. Let $d_{WV}$ denote the component mapping of $d$ mapping to $V$ from $W$, and similarly for other components.  If $d_{WV}$ is an isomorphism, then $(C,d)$ is chain homotopy equivalent to $(C',d')$ with $$d' = d_{C'C'}-d_{C'V}d_{WV}^{-1}d_{WC'}.$$
\end{lemma}
\begin{proof}
Let $f: C' \rightarrow C$, $g: C \rightarrow C'$, and $h: C \rightarrow C$ be defined by

$$f = \iota_{C'} - d_{WV}^{-1} d_{WC'}, \qquad g = \pi_{C'} -  d_{C'V} d_{WV}^{-1}, \mbox{\;\;\;and\;\;\;} h = d_{WV}^{-1}, $$
where $\iota$ and $\pi$ denote inclusion and projection with respect to the direct sum decomposition of $C$. The map $f$ is an isomorphism onto its image, since the second term in $f$ merely adds a $V$-component. The image of $f$ turns out to be a subcomplex, and the new differential $d'$ is merely the pullback of $d$ along $f$.

We claim that  $f$ and $g$ are mutually inverse chain homotopy equivalences between $(C,d)$ and $(C',d')$. Specifically, the following four equations hold:
$$f d' = d f \qquad g d = d' g \qquad \mathbb{I}_{C'} = gf  \qquad \mathbb{I}_C = fg  + hd + dh$$
Verifying these is a routine exercise in applying the identities contained in the equation $d^2 = 0$, such as
\[ d_{WV}d_{VV} + d_{WC'}d_{C'V}+d_{WW}d_{WV} = 0. \qedhere \]
\end{proof}

If the complex $(C,d)$ is filtered and the cancelled map, $d_{WV}$ above, preserves filtration degree, then $d'$ will respect the induced filtration on $C'$ and the maps $f$ and $g$ will be filtered chain homotopy equivalences. This will be our situation in each of the Reidemeister moves below.

	\begin{proposition}\label{ReidInvProp}
  Let $D$ and $D'$ be two diagrams for a link $L$ related by a Reidemeister move of type I, II, or III. 
  Fix an $R$-weighting $w$ for $L$ and a sign assignment $s$ for the diagram $D$. 
  Then there exists a sign assignment $s'$ for the diagram $D'$ which agrees with the sign assignment for $D$ at all crossings uninvolved in 
  the Reidemeister move, and the complexes $C(D, w,s)$ and $C(D', w,s')$ are chain homotopy
  equivalent as $\ell$-graded, $q$-filtered complexes.
\end{proposition}

In Section~\ref{sec:ourconst}, we saw that different sign assignments produce isomorphic complexes. Since any two diagrams for a link are related by a sequence of Reidemeister moves, 
this proposition implies that that the $\ell$-graded $q$-filtered chain homotopy type of the complex $C(D, w, s)$ is also independent of the choice of planar diagram, and hence an invariant of the $R$-weighted link $(L,w)$. This establishes that the associated spectral sequence, called $E_k(L,w)$ in Theorem~\ref{SseqThm}, is an invariant of $(L,w)$.

\begin{proof}

The proof for each of the three Reidemeister moves is similar. We first decompose the complex into
 summands sitting over each of the $2^k$ different resolutions of the crossings implicated in the $k$-th move.
One of these resolutions contains an isolated circle, and we split the complex over that resolution further 
according to whether or not the monomial contains that circle. We then identify two summands $V$ and $W$ for which
$d_{WV}$ is a $q$-grading-preserving isomorphism, and apply the cancellation lemma.

\begin{figure}
  \begin{center}
    \includegraphics{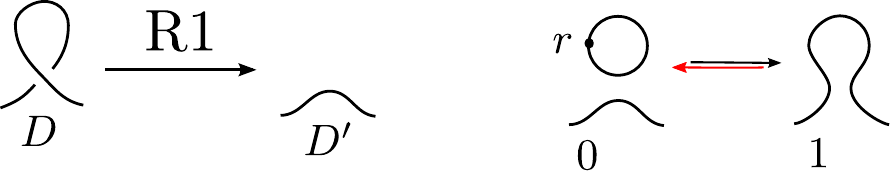}
  \end{center}
  \caption{Left is the first Reidemeister move R1.  Right is chain complex for the diagram $D$, split into two summands corresponding to the two resolutions of the pictured crossing.\label{R1fig}}
\end{figure}

\textbf{R1}
Consider two diagrams $D$ and $D'$ for a link $L$ in Figure~\ref{R1fig}. 
Let $s$ be a sign assignment for $D$. It can be verified easily that the restriction of $s$ to the
vertices of the diagram for $D'$ yields a valid sign assignment $s'$.

Let $(C,d)$ be the complex $C(D,w,s)$, and let $(C',d')$ be the complex $C(D',w,s')$. 
Let $C_0$ be the summand of $C$ corresponding to complete resolutions which include the $0$-resolution at the pictured crossing, and let $C_1$ be the summand of $C$ corresponding to complete resolutions which include the $1$-resolution at the pictured crossing. Let $C_0^-$ and $C_0^+$ be the summands of $C_0$
  spanned by monomials divisible and not divisible, respectively, by
  the circle $x_r$ corresponding to the pictured circle.
  
 Since the component of $d$ mapping from $C_0^+$ to $C_1$ is just merging in the $1$ on the pictured circle, it is an isomorphism. Hence we may apply the cancellation lemma with with $V = C_0^+$ and $W=C_1$. Since $C_0^+$ and $C_0^-$ have the same resolution at the pictured crossing, there is no component of $d$ mapping from one to the other. Hence the new complex is just $C_0^-$ with the restriction of the original differential. Since extra circle never interacts with the remainder of the diagram for $L$, this complex $(C_0^-,d)$ is isomorphic to the post-move complex $(C',d)$.

\begin{figure}
  \begin{center}
    \includegraphics{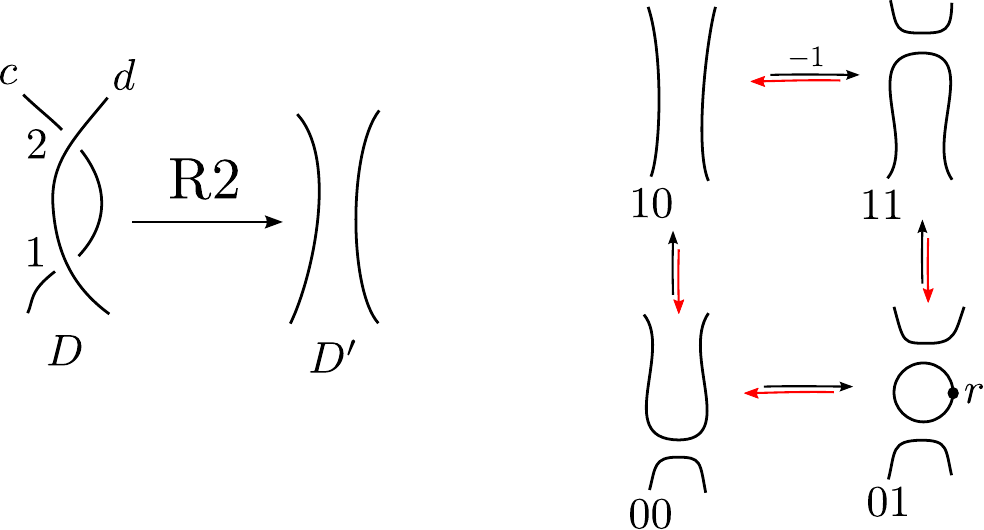}
  \end{center}
  \caption{Left is the second Reidemeister move R2.  Right is chain complex for the diagram $D$, split into four summands corresponding to the resolutions of the pictured crossings.\label{R2fig}}
\end{figure}

\textbf{R2}
Consider two diagrams $D$ and $D'$ for a link $L$ in Figure~\ref{R2fig}. 
Let $s$ be a sign assignment for $D$. It can be verified easily that the restriction of $s$ to the
vertices of the diagram for $D'$ yields a valid sign assignment $s'$.

Let $D_{ij}$ with $i, j\in \{0, 1\}$ denote
  the diagrams obtained by resolving the crossings involved in the
  Reidemeister move in $D$.  Let $C_{ij} = C(D_{ij}, w,s)$. Let
  $C_{01}^-$ and $C_{01}^+$ be the summands of $C_{01}$ spanned by
  generators divisible and not divisible, respectively, by the circle
  $x_r$ corresponding to the pictured circle. The four
  summands $C_{00}, C_{11}, C_{01}^+$ and $C_{01}^-$ are all
  naturally isomorphic, and the summand $C_{10}$ is isomorphic to the post-move complex $C' = C(D',w,s')$.
  
 We will apply the cancellation lemma with $V = C_{00}\oplus C_{01}^+$ and 
  $W = C_{01}^-\oplus C_{11}$. The component of $d$ from $V$ to $W$ is just the original Khovanov differential $d_0$, and it is block diagonal: $C_{01}^+$ maps to $C_{11}$ isomorphically (merging in a $1$) and $C_{00}$ maps to $C_{01}^-$ isomorphically (splitting of an $x$).
  
  The cancelled complex is just $C_{10}$, with differential 
  
  $$d_{C_{10}C_{10}} - d_{C_{10}V}d_{WV}^{-1}d_{WC_{10}}.$$
  
  But $d_{WC_{10}}$ lands on $C_{11}$, which is carried to $C_{01}^+$ by $d_{WV}^{-1}$, and $d$ has no component from $C_{01}^+$ to $C_{10}$. Hence the new differential is just the restriction of the old, and we have
  
  $$(C,d) \cong (C_{10}, d \vert_{C_{10}}) \cong (C',d')$$
  
\begin{figure}
  \begin{center}
    \includegraphics{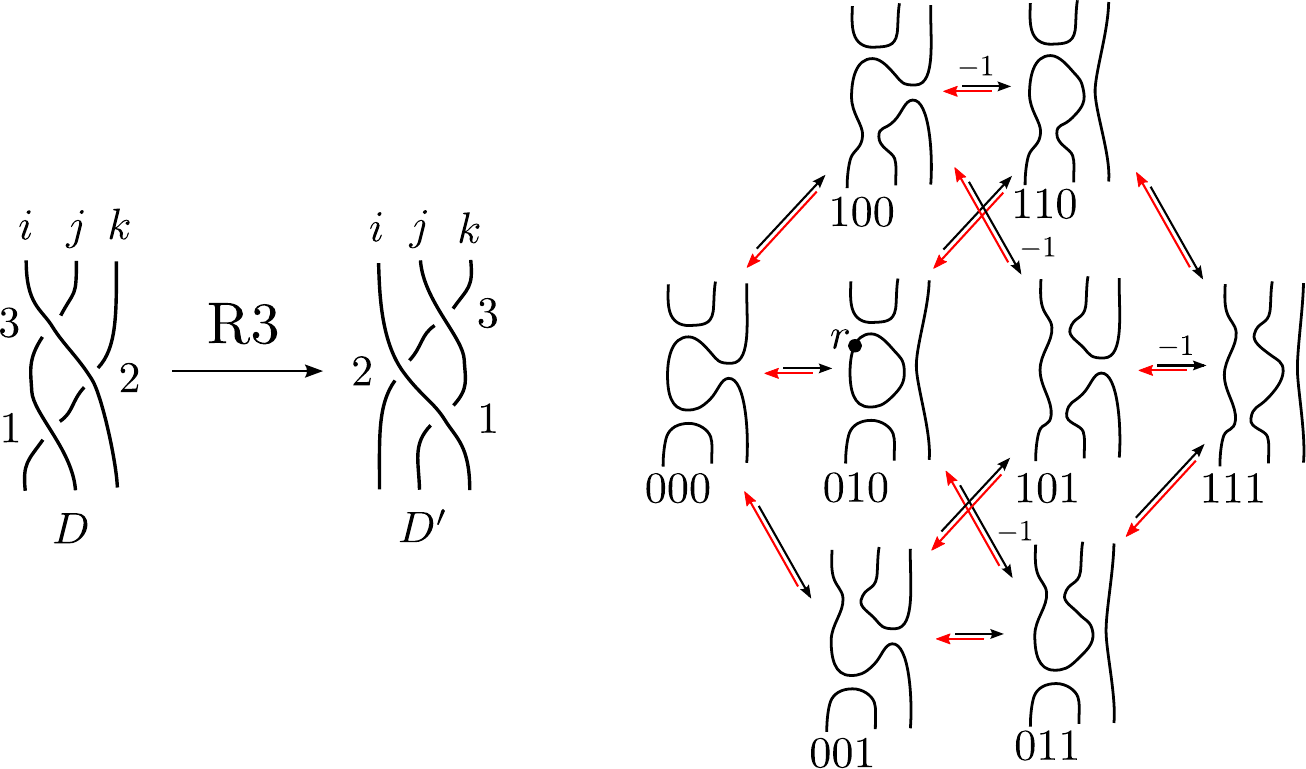}
  \end{center}
  \caption{Left is the third Reidemeister move R3.  Right is chain complex for the diagram $D$, split into eight summands corresponding to the resolutions of the pictured crossings.\label{R3fig}}
\end{figure}

\textbf{R3}
Reidemeister 3 is more complicated, and we must keep track of the signs in Khovanov's cube, the sign assignment $s$, and the weights.

Consider the diagrams $D$ and $D'$ in Figure~\ref{R3fig}. Label the strands $i$, $j$, and $k$, from left to right along the top of $D$. Denote by $w_i, w_j,$ and $w_k$ the weights of their components. Order the crossings up the page $1$, $2$, and $3$. Using Khovanov's sign assignment, the edges in the cube of resolutions for $D$ labeled $-1$ in the figure have a negative sign in the differential: $(-1)^{n(I, J)} = -1$. ($100\leftrightarrow 110, 100\leftrightarrow 101, 010\leftrightarrow 011, 101\leftrightarrow 111$.)

Choose a sign assigment $s$ for $D$ such that
\[ s(1) = s(3) = 1 \text{ and } s(2) = -1. \]
A choice of sign at one crossing determines the sign assignment on that component of the diagram by (\ref{SignEq}).  Take the sign assignment $s'$ for $D'$ which agrees with $s$ on the crossings not implicated in the Reidemeister move. Again, (\ref{SignEq}) implies
\[ s'(1) = s'(3) = -1 \text{ and } s'(2) = 1. \]

Let $(C, d) := C(D, w, s)$ and $(C', d') := C(D', w, s')$.  The weights
$c(j) = s(j)(w_\text{over} - w_\text{under})$ of the reverse edge maps
in $d_1$ evaluate to 
\begin{align*}
  c(1) &= w_j - w_k \\
  c(2) &= w_k - w_i \\
  c(3) &= w_i - w_j,
\end{align*}
at the three pictured crossings, and the weights $c'(j)$ in $d'_1$ are
\begin{align*}
  c'(1) &= w_j - w_i \\
  c'(2) &= w_i - w_k \\
  c'(3) &= w_k - w_j.
\end{align*}

First, we will simplify the complex $(C,d)$. As in the previous parts, let $C_{010}^-$ and $C_{010}^+$ be the summands of $C_{010}$
  spanned by monomials divisible and not divisible, respectively, by
  the circle $x_r$ corresponding to the pictured circle.

We apply the cancellation lemma with

$$V = C_{000} \oplus C_{010}^+ \qquad W = C_{010}^- \oplus C_{011}.$$
The component of $d$ from $V$ to $W$ is just the Khovanov differential $d_0$, and it is block diagonal: $C_{000}$ maps to $C_{010}^-$ isomorphically (splitting off an $x$) and $C_{010}^+$ maps to $C_{011}$ isomorphically (merging in a $1$, with a minus sign from the cube). The reduced complex will have underlying abelian group 

$$C_{\text{red}} = C_{100}\oplus C_{001} \oplus C_{110} \oplus C_{101} \oplus C_{111}.$$

 After chasing the diagram to find the maps into $V$ and the maps out of $W$, you will find that the correction term $d_{C_{\text{red}}V}d_{WV}^{-1}d_{WC_{\text{red}}}$ has four components.

\begin{eqnarray*}
C_{001} &\xrightarrow{-1} C_{110} \\
C_{111} &\xrightarrow{w_k-w_j} C_{110} \\
C_{110} &\xrightarrow{w_j-w_k} C_{100} \\
C_{110} &\xrightarrow{w_j-w_k} C_{001}.
\end{eqnarray*}

Each map is induced by the obvious cobordism relating the resolutions, weighted by some element of $R$. Subtracting these from the restriction of the original differential $d$ to $C_{\text{red}}$ yields the complex pictured in Figure~\ref{R3cancelledfig}.  Here, the edge labels give the total coefficient of the forward or reverse edge maps in $d_{\text{red}}$.  The absence of a label on a forward edge maps the coefficient is $+1$.  The label $i - j$, for example, denotes the coefficient $w_i - w_j$.

\begin{figure}
  \begin{center}
    \includegraphics{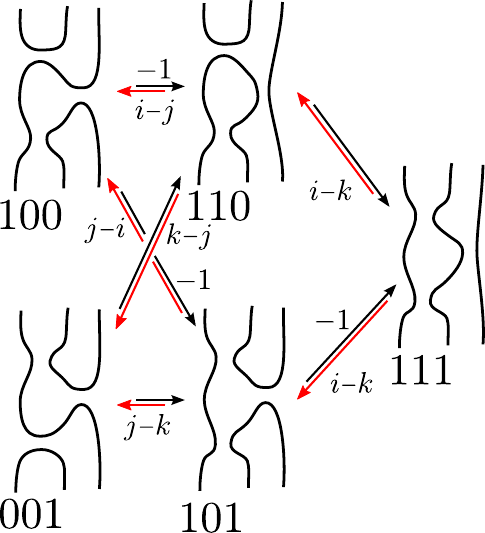}
  \end{center}
  \caption{A reduced chain complex for the diagram $D$. \label{R3cancelledfig}}
\end{figure}

\begin{figure}
  \begin{center}
    \includegraphics{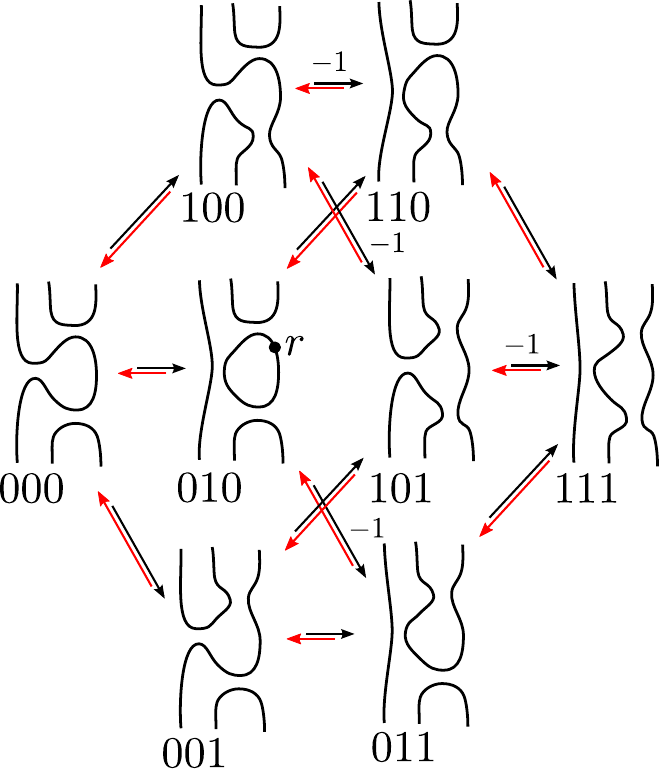}
  \end{center}
  \caption{The chain complex for the post-R3 diagram $D'$. \label{R3postfig}}
\end{figure}

\begin{figure}
  \begin{center}
    \includegraphics{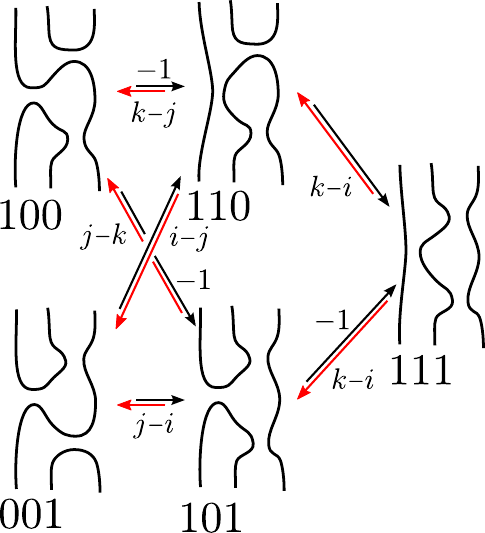}
  \end{center}
  \caption{A reduced chain complex for the post-R3 diagram $D'$. \label{R3postcancelledfig}}
\end{figure}

The complex $(C',d')$ can be simplified using a similar cancellation. The relevant resolutions are drawn in Figure~\ref{R3postfig}. Apply the cancellation lemma with
$$V = C'_{000} \oplus C'^+_{010}  \qquad W = C'^-_{010}  \oplus C'_{011}.$$
The resulting complex, $C'_{\text{red}}$ is pictured in Figure~\ref{R3postcancelledfig}. It contains all the same resolutions as $C_{\text{red}}$, the only difference is that all of the maps between pictured summands have reversed signs. The map $\phi: C_{\text{red}} \rightarrow C'_{\text{red}}$, defined by

\begin{eqnarray*}
C_{100} &\xrightarrow{1} C'_{001} \\
C_{001} &\xrightarrow{1} C'_{100} \\
C_{110} &\xrightarrow{-1} C'_{101} \\
C_{101} &\xrightarrow{-1} C'_{110} \\
C_{111} &\xrightarrow{1} C'_{111}
\end{eqnarray*}

is an invertible chain map. The sequence $C \cong C_{\text{red}} \cong  C''_{\text{red}} \cong C'$ yields the desired isomorphism for diagrams related by Reidemeister 3.
\end{proof}

	We can now prove that the total homology of the complex for a link is just the Khovanov homology of the disjoint union of its components. This completes the proof of Theorem~\ref{SseqThm}

\begin{theorem}
Let $(L,w)$ be an $R$-weighted link, and suppose that for each pair of components $i$ and $j$ with distinct weights, the difference $w_i - w_j$ is invertible in $R$.  Let $D$ be any diagram for $L$. Let $L^{(r)}$ denote the sublink of $L$  consisting of those components with weight $r$. Then the spectral sequence converges to

  \[ H^*(C(D, w)) \iso Kh^*\left ( \coprod_{r\in R} L^{(r)}; R \right ) \]
  
\end{theorem}

\begin{proof}
Choose an arbitrary ordering $\succ$ on the set ${w_1,\dots,w_n}\subset R$ of weights. By Proposition \ref{CrossingChangeProp}, changing a crossing between components with distinct weights will produce a chain complex $C(D',w)$ with the same $l$-graded total homology. So we may change crossings until each component $i$ lies entirely over component $j$ whenever $w_i \succ w_j$. This produces a diagram $D'$ for some link $L'$, whose sublinks are still the $L^{(r)}$, now completely unlinked from one another. By repeated application of Reidemeister moves 1 and 2, we may slide these components off of one another until we get a diagram $D''$ for $L'$ with no crossings between $L^{(r)}$ and $L^{(r')}$ for $r \neq r'$. The differential for $C(D'',w)$ is the same as Khovanov's differential, since $d_1=0$, and $L'$ is just the disjoint union of the sublinks $L^{(r)}$.
\end{proof}

We can now give a stronger version of the rank inequality Corollary~\ref{CorRankBounds}.

\begin{corollary}
Let $\F$ be any field, and let $L$ be a link with components $K_1,\dots,K_m$. Then
\[ \rank^{\ell} Kh^*(L;\F) \geq \rank^{\ell+t} \otimes_{c=1}^m Kh^*(K_c;\F), \]
where each side is $\ell$-graded and the shift $t$ is given by
\[t =  \sum_{c<d}  2 \text{\emph{lk}}(L_c,L_d). \]
\end{corollary}
\begin{proof}
Assume for the moment that the field $\F$ has more elements than $L$ has components, so we
can weight each component by a different element $w_c\in \F$. Then all differences will be invertible, so the above theorem characterizing the abutment of the spectral sequence applies. That would give an inequality of total ranks. To see the $\ell$-gradings, we need to compute the grading shift in the isomorphism relating $C(D,w)$ and $C(D',w)$. Recall the formula for the $\ell$ grading:
\[ \ell(x) = \deg(x) - p(I)  -  \mbox{writhe}(D). \] 
For a fixed monomial $x$ over a fixed resolution $I$, the terms $\deg(x)$ and $p(I)$ are the same before and after a crossing change; only the writhe differs. Each time we change a crossing between components $c$ and $d$, the writhe will shift by $\pm 2$ and the linking number $\mbox{lk}(L_c,L_d)$ will shift by $\pm 1$. (The linking numbers with other components remain unchanged.) Thus
\[ \ell(x) + \sum_{c<d} 2\mbox{lk}(L_c,L_d) = \ell(x') +  \sum_{c<d}  2\mbox{lk}(L'_c,L'_d), \]
where $x'$ is the same monomial viewed as a generator of $C(D',w)$, and $L'_c$ is the component of $L'$ which $L_d$ turns into. But the components of $L'$ are unlinked, so we ultimately have
\[ \ell(x') = \ell(x) +  \sum_{c<d}  2\mbox{lk}(L_c,L_d). \]

Now we address the size of $\F$. Since the differential in the chain complex computing $Kh(L)$ uses only $\pm 1$ coefficients, its rank is the same after a field extension. We may take a suitably large extension $\F'$ of $\F$, run the above argument for some choice of weights, and then note that $\rank_\F' Kh(L;\F') = \rank_\F Kh(L; \F)$.
\end{proof}

\section{Properties of spectral sequences}
\label{sec:ssreview}
	
We offer a quick review of spectral sequences, following
Serre \cite{serre}.  Let $(C,d)$ be a finitely generated chain
complex. A filtration $\calF$ on $C$ is an assignment to each element
$x\in C$ a filtration degree $p(x)\in\Z\cup\{-\infty\}$ such that
$p(x-y)\leq\max(p(x),p(y))$ and $p(dx)\leq p(x)$. We will occasionally
write $C^k$ for the $k^\text{th}$ piece of the filtration
$\calF^kC=\{x\in C\vert p(x)\leq k\}$. Homological algebra usually
concerns cycles and boundaries. The filtration provides notions of
approximate cycles and early boundaries:
\begin{eqnarray*}
Z_{r}^{k}&=&\left\{ x\in C^{k}\vert dx \in C^{k-r} \right \}\\
B_{r}^{k}&=&\left\{ dy\in C^{k}\vert y\in C^{k+r}  \right\}.
 \end{eqnarray*}
The spectral sequence corresponding to the filtration is a sequence
of chain complexes $\left(E_{r}^{k},d_{r}\right)$, called \emph{pages},
defined by 
\[
E_{r}^{k}=\slfrac{Z_{r}^{k}}{Z_{r-1}^{k-1}+B_{r-1}^{k}}.
\]

If $x$ is in $Z_{r}^{k}$, then $dx$ is in $Z_{r}^{k-r}$: by definition $dx\in C^{k-r}$, 
and $d(dx)=0$. The differential on $E_{r}^{k}$
is then given by taking the equivalence class: $d_{r}[x]:=[dx]$.
The remarkable property of this sequence is that each page is the
homology of the previous one: $E_{r+1}^{k}=H_{*}(E_{r}^{k},d_{r})$.

A spectral sequence is said to \emph{collapse on page} $l$ if $d_{r}=0$
for all $r\geq l$.

Since $C$ is finitely generated, there is some integer $N$ such that,
for all $r > N$, $Z_{r}^{k}$ just consists of all cycles in degree
$\leq k$ and $B_{r}^{k}$ consists of all boundaries in degree $\leq k$
(that is, $Z_{r}^{k}=Z^{k}$ and $B_{r}^{k}=B^{k}$). The quotient
$\slfrac{Z^{k}}{B^{k}}$ is not the homology of the $k^\text{th}$
filtered piece $C^{k}$, because $B^{k}$ consists of elements of
$C^{k}$ which are boundaries in $C$, not just boundaries of elements
in $C^{k}$. In fact, the quotient is\[
\slfrac{Z^{k}}{B^{k}}\cong i_{*}H_{*}(C)\]
where $i:C^{k-1}\hookrightarrow C$ denotes the inclusion of the $k^\text{th}$
filtered piece into the total complex. For all $r>N$, then, we have
\begin{eqnarray*}
E_{r}^{k} & = & \slfrac{Z^{k}}{Z^{k-1}+B^{k}}\\
 & = & \left(\slfrac{Z^{k}}{B^{k}}\right)\slash\left(\slfrac{Z^{k-1}}{B^{k-1}}\right)\\
 & = & \slfrac{i_{*}H_{*}(C^{k})}{i_{*}H_{*}(C^{k-1})}.
\end{eqnarray*}

We denote this stable page by $E_{\infty}^{k}$, and observe that it is
the associated graded group of the total homology $H_{*}(C)$ by the
filtration
\[\mathcal{G}^{k}H_{*}(C)=i_{*}H_{*}(C^{k}).\]
In particular, the total rank of the $E_{\infty}$ page is independent
of the choice of filtration: 
\[ \sum_{k}\rank E_{\infty}^{k}=\rank H_{*}(C).\]
In contrast, the time of collapse does depend on the choice of filtration,
though in a controlled way. (We doubt that the following proposition
is original, but were unable to find it in the literature.)

\begin{proposition}\label{filtrationchangeprop}
Let $(C,d)$ be a finitely generated chain complex, with two different
filtrations $\calF$ and $\calF^{\prime}$ which are close in the following
sense: for any $x\in C$, the difference in filtration degree $p'(x)-p(x)$
is either $0$ or $1$. Then the $p$-spectral sequence collapses
at most one page after the $p'$-spectral sequence does.\end{proposition}
\begin{proof}
Say that the $p'$-spectral sequence has collapsed by the $(r-1)^{st}$
page. We want to show that any class $[x]\in E_{r}^{k}$ must have
$d_{r}[x]=0\in E_{r}^{k-r}$, for then the $p$-spectral sequence will have collapsed
on page $r$.

Suppose for the sake of contradiction that there is
some $x\in Z_{r}$ such that $[x]\in E_{r}$ has nonzero differential.
Without loss of generality, we may take the chain $x$ with minimal
$p'(x)+p'(dx)$. Let $k$ be the degree $p(x)$, so $x \in Z^k_r$.
If $p(dx)<k-r$, then $dx\in Z_{r-1}^{k-r-1}$ and
$[dx]_r$ would represent $0$ in $E_{r}^{k-r}$. Since $d_r[x] = [dx]$ is
nontrivial, we must have $p(dx) = k - r$.

We now consider the $p'$-degrees of all the elements. Let $k'=p'(x)$
and $r'=p'(x)-p'(dx)$. Note that \begin{eqnarray*}
r' & = & p'(x)-p'(dx)\\
 & = & p'(x)-p(x)-(p'(dx)-p(dx))+p(x)-p(dx)\\
 & \in & \{0,1\}-\{0,1\}+r\\
 & \geq & r-1\end{eqnarray*}
Since the $p'$-spectral sequence, whose pages we will denote $E_{*}^{*}(p')$,
has collapsed by page $r-1$, it has also collapsed by page $r'$.
And by construction, $x$ represents a class in $E_{r'}^{k'}(p')$.
Post-collapse, the differential is identically zero, so $d_{r'}[x]_{p'}$
must represent zero in $E_{r'}^{k'-r'}(p')$. In terms of chains,
this means that
\[ dx=w+dz \]
for some $w\in Z^{k'-r'-1}_{r'-1}$ with with $p'(w) \leq k' - r' - 1$
and some $dz\in B^{k'-r'}_{r'-1}$ with $p'(z) \leq k' - 1$.  Since
$p$-gradings are at most one less than $p'$-gradings, $p(x) \geq k' -
1$ and $p(dx) \geq k' - r' -1$. Consequently, $p(z)\leq p(x)$ and
$p(w)\leq p(dx)$.

Since $dw = ddx - ddz = 0$, we have that $w\in Z_r^{k-r}$.
Since $dz = dx - w$, we have $p(dz) \leq \max(p(dx), p(w))$, and $z \in Z_r^k$.

We break into two cases.

\noindent \textbf{Case 1:} $[w] = 0 \in E_r^{k-r}$.

Set $\overline{x}=z$. Then $[\overline{x}]$ is a class in $E_r^k$ with 
\[ d_r[\overline{x}]=[dz]=[dz]+[w]=[dx]\neq0. \]
But $p'(\overline{x})=p'(z)<p'(x)$ and $p'(d\overline{x})=p'(dx-w)=p'(dx)$,
violating minimality.

\noindent \textbf{Case 2:} $[w] \neq 0 \in E_r^{k-r}$.

Set $\overline{x}=x-z$. Then $[\overline{x}]$ is a class in $E_r^k$ with
\[ d_r[\overline{x}]=[dx - dz]=[w] \neq 0. \]
But $p'(\overline{x})=p'(x - z) \leq p'(x)$ and $p'(d\overline{x})=p'(w)<p'(dx)$, violating minimality.
\end{proof}

\subsection{Endomorphisms of spectral sequences} \label{sseqend}

Suppose that $f$ is an endomorphism of the filtered chain complex
$C$ which shifts filtration degree by $l$,\[
p(fx)=p(x)-l\;\;\forall x\in C.\]
 Then $f$ \emph{acts} on the spectral sequence the following sense
\begin{enumerate}
\item There is an endomorphism $f_{r}$ of the $r^\text{th}$ page given by \begin{eqnarray*}
f_{r}:E_{r}^{k} & \rightarrow & E_{r}^{k-l}\\{}
[x] & \mapsto & [fx]\end{eqnarray*}
This is well-defined: since $f$ shifts $p(dx)$ by the same amount
that it shifts $p(x)$, it takes $Z_{r}^{k}$ into $Z_{r}^{k-l}$
and $B_{r}^{k}$ into $B_{r}^{k-l}$.
\item The action of $f_{r+1}$ on $E_{r+1}$ is the same as the one induced
by $f_{r}$ on the homology of $(E_{r},d_{r})$.
\item The action of $f_{\infty}$ on $E_{\infty}$ is the associated graded
action of \[
f_{*}:H_{*}(C)\rightarrow H_{*}(C)\]
with respect to the filtration $\mathcal{G}$ above. That is, if $[x]\in\mathcal{G}^{k}=i_{*}H_{*}(C^{k})$
is represented by $x\in Z^{k}$, then $fx\in Z^{k-l}$ and the image
$f_{*}[x]=[fx]$ lies in $\mathcal{G}^{k-l}.$ Moreover, $x$ also
serves as a representative of the equivalence class of $[x]_{\infty}\in E_{\infty}^{k}=\slfrac{\mathcal{G}^{k}}{\mathcal{G}^{k-1}}$
and $f_{\infty}[x]_{\infty}=[fx]_{\infty}.$
\end{enumerate}
We will later encounter a spectral sequence where we know the action
of an endomorphism $X$ on $H_{*}(C)$ and investigate the possible
associated graded actions on the $E_{\infty}$ page.

\section{The splitting number}
\label{sec:splitnum}
	The unknotting number of a knot is the minimum number of times the knot must be passed through itself to untie it. It is an intuitive measure of the complexity of a knot, though strikingly difficult to compute. We would like to suggest a similar number measuring the complexity of the linking \emph{between} the components of a link, unrelated to the knotting of the individual components.

\begin{definition}
The splitting number of a link $L$, written $\Sp(L)$, is the minimum number of times the different components of the link must be passed through one another to completely split the link. Equivalently, $\Sp(L)$ is the is the minimum over all diagrams for $L$ of the number of between-component crossings changes required to produce a completely split link.
\end{definition}

\begin{figure}[H]
  \begin{center}
    \includegraphics[scale=0.3]{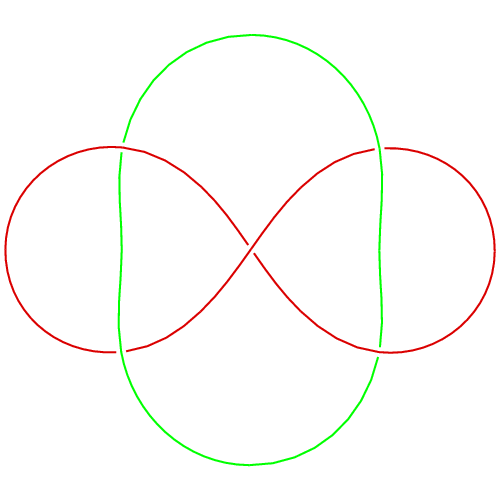}
  \end{center}
  \caption{The Whitehead link has splitting number 2.\label{Lwhitehead}}
\end{figure}

A completely split link has splitting number $0$. The Hopf link has splitting number $1$, as demonstrated by the standard diagram. In general, any diagram for a link $L$ gives an upper bound on $\Sp(L)$, as one may change crossings until the components of the link are layered one atop the next.

The Whitehead link $L_W$ has splitting number $2$---change two diagonally opposite crossings  in the standard diagram (Figure~\ref{Lwhitehead}). While changeing the crossing in the center would split the link,  that crossing is internal to one component so not allowed. To see that $\Sp(L_W)\neq 1$, note that a crossing change between components $K_c$ and $K_d$ of a link $L$ changes the linking number $\lk(K_c,K_d)$ by $\pm 1$. Since the Whitehead link has linking number $0$, an even number of crossing changes will be required. 

If $L$ is a two-component link with components $K_1$ and $K_2$, then the quantity
\[b_{lk}(L) := \begin{cases}
\left \vert \lk(K_1,K_2) \right \vert & \mbox{ if }L\mbox{ is non-split and }\lk(K_1,K_2)>0\\
2 & \mbox{ if }L\mbox{ is non-split and }\lk(K_1,K_2)=0\\
0 & \mbox{ if }L\mbox{ is split}\end{cases} \]
provides a lower bound on $\Sp(L)$. If $L$ has many components, we define
\[ b_{lk}(L) := \sum_{c<d} b_{lk}(L_{cd}),
\]
where $L_{cd}$ denotes the sublink consisting of the $c^\text{th}$ and $d^\text{th}$ components. Since splitting a link certainly requires that one change enough crossings to split each pair of components (and each crossing implicates only one two-component sublink), we conclude that
\[\Sp(L) \geq b_{lk}(L).
\]

Our spectral sequence provides less obvious lower bound for the splitting number: the splitting number plus one is at least the index of the page on which the spectral sequence collapses.

\begin{reptheorem}{SplittingNumThm}
Let $L$ be a link and let $w_c\in R$ be a set of component weights such
  that $w_c - w_d$ is invertible for each pair of components $c$ and $d$. Let $b(L,
  w)$ be largest integer $k$ such that $E_k(L, w) \neq E_\infty(L, w)$.
  Then $b(L, w)\le \Sp(L)$.

\[ \Sp(L) \geq b(L,w) \]
\end{reptheorem}

\begin{proof}
We proceed by induction on splitting number. If $L$ is a split link, then there is a diagram in which $d_1 = 0$ so the spectral sequence collapses immediately: $E_1 = E_\infty$ and $b(L) = 0$.

If $L$ is non-split, then there is a diagram $D$ in which changing exactly $k = \Sp(L)$ crossings produces a diagram for a split link. Consider the diagram $D'$ resulting by changing just one of those crossings, say $i$; the link $L'$ depicted will have splitting number $k-1$.

In the proof of Proposition~\ref{CrossingChangeProp}, we saw that the two filtered chain complexes $C(D,w)$ and $C(D',w)$ can actually be viewed as two filtrations $\calF$ and $\calF'$ on, say, $C(D,w)$, after rescaling the generators of $C(D',w)$ by units in $R$. These two filtrations by gradings $g$ and $g'$, differ in a controlled way.

Recall that for a generator $x$ of $C(D,w)$, the relevant gradings are
\[q(x) = \deg(x) +p(I) + |I| + n_{+} (D) - 2 n_{-} (D) \mbox{\;\;and\;\;} g(x) = \frac{q(x) - |L|}{2}.\]
The monomial degree $\deg(x)$ and circle count $p(I)$ are the same in both $D$ and $D'$. If $x$ sits over the $0$-resolution of the crossing $i$ in $D$, then it sits over the $1$-resolution of $i$ in $D'$, and vice versa.  So the value of $|I|$ differs by $\pm1$ between the two complexes. Finally, the difference $n_{+} (D) - 2 n_{-} (D)$ decreases (increases) by $3$ if $i$ is a positive (negative) crossing in $D$. Thus the difference in filtration degree $g'(x) - g(x)$ is in $\{-1,-2\}$ if the crossing is positive and $\{1,2\}$ if the crossing is negative.

Since a global shift in filtration degree does not affect the page at which the corresponding spectral sequences collapses, Proposition~\ref{filtrationchangeprop} applies. We conclude that the spectral sequence for $L$ collapses at most one page after the spectral sequence for $L'$, so 
\[ b(L,w) \leq b(L',w) + 1 \leq \Sp(L') + 1 = \Sp(L).\qedhere \]
\end{proof}

An interesting example is the link $L = {}^2n^{13}_{8862}$ shown in Figure~\ref{boundexfig}. The two components are a trefoil and the unknot, and they have linking number $1$. There is an obvious way to split the $L$ by changing three crossings, say, pulling the red component on top of the green one. The spectral sequence with the nontrivial $\F_2$ weighting $w$, shown in Table~\ref{boundextab}, collapses on the $E_3$ page, so $\Sp(L) \geq b(L,w) = 2$. Since $\Sp(L)$ must have the same parity as the linking number, we have that $\Sp(L) = 3$.

The calculation of the spectral sequence for ${}^2n^{13}_{8862}$ and many other links is discussed in Section~\ref{sec:ex}.

\footnotesize % \tiny \scriptsize \footnotesize \small
\begin{center}
\begin{longtable}{p{0.15\textwidth}@{}p{0.03\textwidth}p{0.07\textwidth}p{0.75\textwidth}}
\caption{$E_k({}^2n^{13}_{8862}, w)$ over $\F_2$ with non-trivial weight function $w$.  $E_1(L, w) = Kh(L)$
  omitted.}\label{boundextab}\\*
\toprule
Link $L$ & $E_k$ & $\rank E_k$ & $P_k(q, t) = \sum_{i,j} (\rank E_k^{ij})t^iq^j$ \\
\midrule
${}^2n^{13}_{8862}$ & $E_2$ & 20 & $t^{-2}q^{-2} + t^{-2} + q^{2} + q^{4} + t^{1}q^{2} + t^{1}q^{4} + t^{2}q^{4} + t^{2}q^{6} + t^{3}q^{6} + t^{3}q^{8} + 2t^{4}q^{8} + 2t^{4}q^{10} + t^{5}q^{10} + t^{5}q^{12} + t^{6}q^{12} + t^{6}q^{14} + t^{7}q^{14} + t^{7}q^{16}$ \\
 & $E_3$ & 12 & $t^{2}q^{4} + t^{2}q^{6} + 2t^{4}q^{8} + 2t^{4}q^{10} + t^{5}q^{10} + t^{5}q^{12} + t^{6}q^{12} + t^{6}q^{14} + t^{7}q^{14} + t^{7}q^{16}$ \\
\bottomrule
\end{longtable}
\end{center}
\normalsize

\section{Detecting unlinks}
\label{sec:unlinkdetect}
	In this section, we work over a field $\F$ of characteristic $2$. Since our construction relies on choosing different weights for different components, $\F_2$ itself is not large enough to accommodate many-component links. The specific choice of a larger field is unimportant, so we will write $\F$ for some finite field of characteristic $2$ with more elements than there are components of the link under consideration. Since $Kh(L; \F) \cong Kh(L; \F_2) \otimes \F$, the rank of Khovanov homology is independent of the choice of $\F$. For this reason, we will often write $Kh(L)$ for $Kh(L;\F)$.

Kronheimer and Mrowka have shown that Khovanov homology detects the unknot. That is, if a knot $K$ has $Kh(K)$ of rank $2$, then $K$ is the unknot.

Corollary~\ref{CorRankBounds} provides an immediate upgrade.

\begin{proposition}\label{PropUnknotComp}
Let $L$ be an $m$-component link, and suppose that the rank of $Kh(L)$ is $2^m$. Then each component of $L$ is an unknot.
\end{proposition}
\begin{proof}
Let $K_1,\dots,K_m$ be the components of $L$. By Corollary~\ref{CorRankBounds}, we have a rank inequality 
$$\rank Kh(L) \geq \rank Kh(K_1)\times \rank Kh(K_2) \times \cdots \times \rank Kh(K_m).$$
The left-hand-side is $2^m$. Since every knot has Khovanov homology of rank at least two, the right-hand side is at least $2^m$. Hence every one of the components $K_i$ must have $\rank (Kh(K_i))=2$. By Kronheimer and Mrowka's result, each of those components is an unknot.
\end{proof}

Equality is possible: the Hopf link has rank four, just like the two-component unlink. This generates a family of such examples: iterated connect-sums and disjoint unions of Hopf links and unknots. The resulting links can be described as forests of unknots: given a (planar) forest $F$, form a link $L_F$ by placing an unknot at each vertex then clasping them together along each edge (Figure ~\ref{ForestOfUnknotsFig}). By \cite{shum}, we have $\rank Kh(L_F) = \rank Kh(U^m)$.

\begin{question}
Are forests of unknots the only $m$-component links with Khovanov
homology of rank $2^m$ over $\F_2$?
\end{question}

\begin{figure}
  \begin{center}
    \includegraphics{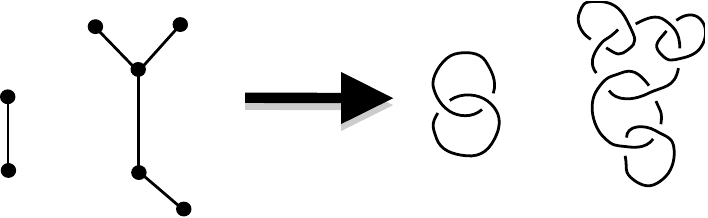}
  \end{center}
  \caption{A forest $F$ gives rise a link $L_F$ whose Khovanov homology has the same rank as that of the unlink. \label{ForestOfUnknotsFig}}
\end{figure}

None of these nontrivial links have the same bigradings at the unlink. As we will show later, this is no coincidence.

	\subsection{The Khovanov module}
		
Khovanov homology is not just an abelian group: $Kh(L)$ a module over the component algebra
\[ A_m = \F_2[X_1,\dots,X_m]/(X_1^2, X_2^2, \dotsc, X_m^2), \]
see \cite{heddenni}.  The module structure is defined by choosing
marked points $p_c$ on each component $c$ of $L$.  Then $X_c$ acts by
$X_{p_c}$.

In fact, this module structure extends to all the pages $E_k$.  The
map $X_{p_c}$ shifts the $g$ gradings by $-1$, so it preserves the
filtration $\calF^p$.  It remains to show that $X_p$ for a marked
point $p$ is a chain map with respect to the total differential $d$
and that the module structure induced on $E_k$ is independent of the
choice of marked points.

\begin{proposition}
  Let $p$ be a marked point on $D$ away from the double points.  Then
  we have that $dX_p = X_pd$.
\end{proposition}
\begin{proof}
  $X_p$ commutes the Khovanov edge maps; this is the standard proof that it commutes
  with $d_0$.  The deformation $d_1$ is also sum of edge maps, so the
  proposition follows.
\end{proof}

\begin{figure}
  \begin{center}
    \includegraphics{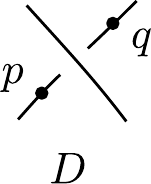}
  \end{center}
  \caption{Moving a marked point across a crossing.\label{moveptsfig}}
\end{figure}

\begin{proposition}
  Let $p$ and $q$ be marked points on either side of crossing $i$ as
  shown in Figure~\ref{moveptsfig}.  Then $X_p$ and $X_q$ are chain
  homotopic.
\end{proposition}
\begin{proof}
  We use the same chain homotopy $H$ as in the proof that the Khovanov
  module is well-defined \cite{heddenni}:
  \[ H = \sum_{\substack{\text{resolutions $I$} \\ I(i) = 1}} \calA(J_i, I), \]
  where $J_i$ differs from $I$ solely at $i$.  Hedden and Ni show $X_p +
  X_q = Hd_0 + d_0H$.  It remains for us to show that $Hd_1 + d_1H =
  0$.  This is an immediate consequence of the facts that $H$ and
  $d_1$ both decrease homological grading and that the reverse edge
  maps commute.
\end{proof}

Since we use $\F$ coefficients to define the complex $C(D,w)$, we will first prove results regarding the action of $A^\F_m := A_m \otimes \F$. The Khovanov module of the unknot is just a copy of $A^\F_1$ viewed as a module over itself: $Kh(U)\cong \F[X]/(X^2)$. Disjoint union of links gives tensor products of modules, over the tensor product algebra; in particular, $Kh(U^m) \cong A^\F_m$.

\begin{proposition}\label{totalhomfreerankone}
Let $L$ be a $m$-component link with $\rank Kh(L) = 2^m$. Then if $D$ is any diagram for $L$, then the total homology $H_*(C(D,w))$ is a free rank-one module over the algebra $A^\F_m$.
\end{proposition}
\begin{proof}
  By Proposition~\ref{PropUnknotComp}, the components of $L$ are all
  unknots.  Order the components from $1$ to $m$.  Since we are only
  interested total homology and its module structure, we can ignore
  the $g$-filtration on $C(D, w)$.  We can produce a diagram $D'$ for
  $U^m$ by swapping mixed crossings in $D$ so that at each crossing, the under-strand
  has lower index than the over-strand.  As we saw in the proof of
  Proposition~\ref{CrossingChangeProp}, $C(D, w)$ and $C(D', w)$
  differ only by rescaling generators by elements of $\F$.  The action of $X_{c_p}$
  commutes with rescaling generators. The total homology and $A_m^\F$ action are
  also invariant under Reidemeister and marked point moves.  By such moves, $D'$ can
  be transformed into $D''$, the standard diagram for $U^m$ with no
  crossings: a disjoint collection of circles with marks.  The complex
  for $D''$ has vanishing differential, and $H_*(C(D'', w))$ is
  manifestly a free rank-one $A_m^\F$-module, as desired.
\end{proof}

	\subsection{Proof of Theorem~\ref{ThmUnlinkDetection}}

Hedden and Ni have shown that the module structure of $Kh$ detects the unlink \cite{heddenni}.

\begin{theorem}[Hedden-Ni] Let $L$ be an $m$-component link. If there
is an isomorphism of $A_m$ modules
\[ Kh(L;\F_2) \cong A_m, \]
then $L$ is the unlink.
\end{theorem}

We can deduce the module structure from the bigradings.

\begin{reptheorem}{ThmUnlinkDetection} Let $L$ be an $m$-component link, and $U^m$ the $m$-component unlink.  If
\[ \rank Kh^{i,j}(L;\F_2) = \rank Kh^{i,j}(U^m;\F_2) \]
  for all $i, j$, then $L$ is the unlink.
\end{reptheorem}
\begin{proof}

The Khovanov homology of the unlink is supported entirely in homological grading $0$, where it has rank ${m \choose r}$ in quantum grading $2r-m$. Since our spectral sequence is graded by $g = (q-m)/2$, the group
$$E^{-k}_1(L,w) \cong Kh^{0,m-2k}(L)$$
has rank ${m \choose k}$ for $0\leq k \leq m$.

As described in  Section~\ref{sseqend}, there is a filtration $\calG$ on the total homology $H=H_*(C(D,w))$ with respect to which 
$$E^{-k}_\infty \cong \slfrac{\calG^{-k} H}{\calG^{-k-1} H}.$$
Since the spectral sequence collapses with $E_1 = E_\infty$, this determines the rank of each filtered piece,
$$\rank_\F \calG^{-k} H=  {m \choose k}  + {m \choose k+1}  +\cdots + {m \choose m} $$

Let $I $ be the (maximal) ideal in $A_m^\F$ generated by the $X_i$. The top nonvanishing power of the ideal is $I^m$, which is spanned by the element $X_1X_2\cdots X_m$, and we have $I^{m+1}=0$. Consider the filtration  
$$0 \subset I^m \subset I^{m-1} \subset \cdots \subset I \subset A_m^\F.$$

By Proposition~\ref{totalhomfreerankone}, the total homology is a free rank-one module over $A_m^\F$, generated by some $e\in \calG^0 H \cong H$. Moreover, since each endomorphism $X_i$ lowers the $g$-grading by $1$, it takes $\calG^{-k}$ into $\calG^{-k-1}$.Hence
$$I^{k} e \subset \mathcal{G}^{-k}H$$
for every $0\leq k \leq m$.

Since $e$ is the generator of a free $A_m^\F$-module, we know that $I^k e$ actually has the same rank as $I^k$ itself, which is the same as the rank of $\mathcal{G}^{-k}H$ computed above. Hence $I^k e = \calG^{-k}H.$

The associated graded module is
$$\bigoplus_k I^{k} e/I^{k+1} e \cong \bigoplus_k A_m^\F[k] e,$$
where $A_m^\F[k]$ denotes the linear span of the monomials of degree $k$ in the $X_i$. This isomorphic to $A_m^\F$ itself, viewed as an $A_m^\F$-module. But $E_\infty(L,w) \cong E_1 (L,w) \cong Kh(L)$. So $Kh(L)$ is a free, rank-one $A_m^\F$ module.

More precisely, $Kh(L;\F)$ is a free, rank-one $\F[X_1,\dots,X_n]/X_i^2$-module. To apply Hedden-Ni, and conclude that $L$ is the unlink, we need to show that $Kh(L;\F_2)$ is a free, rank-one $\F_2[X_1,\dots,X_n]/X_i^2$-module. In general, extending the ground field can make a free module out of a non-free one \cite{MO}. This cannot happen for $A_m$-modules, essentially because $A_m$ is a local ring.

Indeed, suppose that $M$ is a module over $A_m$ such that $M_\F = M \otimes \F$ is a free rank-one module over $A_m \otimes \F$. Let $a \in M_\F$ be a generator, so the $\F$-span of $A_m a$ is all of $M_\F$. Now pick some element $b$ of the original module  $M$ such that $b \notin I\cdot M_\F$. Then $b = \alpha(1 + X)a$ where $\alpha \in \F$ and $X \in I$. Because $I$ is nilpotent of order $m+1$, the coefficient $\alpha(1+X)$ is a unit with inverse $\alpha^{-1}(1 - X + X^2 + \cdots \pm X^m)$. Thus $b$ is also a generator for $M_\F$ as a free $A_m\otimes \F$-module. In particular, $b$ is not annihilated by any element of $A_m$. This means that 
$$\rank_{\F_2} A_m b = 2^m = \rank_\F M_\F = \rank_{\F_2} M.$$
Hence $A_m b$ is all of $M$, and $M$ is a free $A_m$-module.
\end{proof}

\noindent {\bf Example.} It is instructive to see where this argument breaks down for the Hopf link. There, $Kh(L) = E_\infty$ has total rank four, with rank-one summands in $g$-degrees $-1,0,1,2$. Thus the filtration of the rank-1 free $A_n$-module  $H^*(C(D,w))$ has ranks
\[ 1 < 2 < 3 < 4. \]

Let $e$ be a generator, and write $x_i\dots x_j$ for $X_i \dots X_j e$. The filtration is 
\[ \langle x_1x_2 \rangle \subset  \langle x_1x_2, x_1 + x_2\rangle \subset  \langle x_1x_2, x_1, x_2 \rangle \subset  \langle x_1x_2, x_1, x_2, 1\rangle \]
The associated graded has a nonstandard module structure: 
\[ \langle a,b \vert X_1a=X_2a, X_1b=X_2b \rangle. \]
In contrast, the two component unlink has a filtration of ranks $1 < 3 < 4$, and the associated graded is isomorphic $A_2$ itself.

\section{Connections with symplectic topology}
\label{sec:symp}
	There may be a deformation of the differential in symplectic Khovanov homology similar to our $d_1$,
as suggested to the first author by P. Seidel. In \cite{seidelsmith}, Seidel
and Smith construct a symplectic fibration over
$\mbox{Conf}^{2n}(\C)$; the monodromy along a braid $\beta\in\pi_{1}(\mbox{Conf}^{2n}(\C))$
based at $t_{0}$ yields a symplectic automorphism $\phi$ of the fiber $(\mathcal{Y}_{n,t_{0}}, \Omega_0).$
A crossingless matching specifies both a Lagrangian $\mathcal{L}\subset\mathcal{Y}_{n,t_{0}}$
and a way to close the braid into a link $\widehat{\beta}$. Symplectic
Khovanov homology $Kh_{\text{symp}}(\widehat{\beta})$ is then defined to
be the Lagrangian Floer homology $HF(\mathcal{L},\phi(\mathcal{L}))$.
The conjectured deformation is given by perturbing the symplectic
form away from exactness in such a way that the monodromy
given by braiding strands marked with different weights is symplectically
isotopic to the identity with respect to the perturbed form $\Omega_{s}$.
The best-understood example of such a phenomenon is the square of a generalized
Dehn twist $\tau$ in dimension four. If $(M^4,\w_0)$ contains a Lagrangian sphere $L$, then there are families of self-diffeomorphisms $\phi^{s}$ and symplectic forms $\w_{s}$ for $M$ such
that
\begin{itemize}
\item $\phi^{s}$ is a $\w_{s}$-symplectomorphism

\item $\phi^{0}=\tau_{L}^{2}$

\item $\phi^{s}$ is $\w^{s}$-symplectically isotopic to the identity
for $s\neq0$,
\end{itemize}
as described in \cite{seidel1}. In this construction, $\int_{L}\w_{s}=s$ is nonzero
when $s$ is nonzero; the sphere $L$ is no longer Lagrangian for the perturbed forms. To get from dimension
$4$ to dimension $4n$, we recall Manolescu's observation \cite{manolescu}
that $\mathcal{Y}_{n,t_{0}}$ is an open subscheme of the Hilbert
scheme $\mbox{Hilb}^{n}(S_{n,t_{0}})$, where $S_{n,t}$ is a family
of symplectic four-manifolds over $\mbox{Conf}^{2n}(\C)$. (In fact, the $S_{n,t}$ are Milnor
fibers of the $A_{2n-1}$-singularity, studied in detail in \cite{khseidel}). The cohomology
of $S_{n,t_{0}}$ is generated by Lefschetz thimbles $\Delta_{1},\dots,\Delta_{2n}$,
one for each strand in the braid, and it is possible that a perturbation
like \[
\w_{s}=\w+s(w_{1}\Delta_{1}+\cdots+w_{2n}\Delta_{2n})\]
would have the desired property. Indeed, the monodromy of a braid
group generator $\sigma_{i}$ is a Dehn twist about a sphere $L_{i}$,
and $\int_{L_{i}}\w_{s}=s(w_{i}-w_{j})$. Perhaps a similar perturbation
could be found for $\mathcal{Y}_{n,t}$.

Such a symplectic analog would share many features of our deformation of Khovanov homology.
Firstly, our perturbation $d_{1}$ respects the grading $\ell=h-q$,
which is the single grading on symplectic Khovanov homology. (This
also suggests a reason that a second grading on $Kh_{\text{symp}}$ has been hard to find:
there may be holomorphic disks between generators of different $q$-gradings,
coming in cancelling pairs which pick up different areas under
$\w_{s}$.) Secondly, the relationship between Khovanov's $d_{0}$
and our perturbation $d=d_{0}+d_{1}$ parallels the relationship between
the Fukaya category of an exact symplectic manifold and the curved
Fukaya category of a nonexact symplectic manifold. Since $d_{0}^{2}=0$
for local reasons, Khovanov was able to define a theory for tangles \cite{khovanov2}:
there is a ring $H^{n}$, and a $(2n,2m)$-tangle yields a complex
of $(H^{n},H^{m})$ bimodules. However $d^{2}=0$ for nonlocal reasons,
and a $(2n,2m)$-tangle yields a curved complex of
$(H^{n},H^{m})$-bimodules with curvature $d^2=w \in Z(H^{n}\otimes (H^{m})^{op})$. ($Z$ denotes the center.) We will construct the deformed theory for tangles in the forthcoming paper \cite{forgetfultang}.

\section{Sample computations}
\label{sec:ex}
	The combinatorial definition of the spectral sequence makes it amenable to computer calculation. We use {\tt knotkit}, a C++ software package for knot
homology computations written by the second author, to compute the spectral sequence for thousands of links \cite{knotkit}.

These computations show that the spectral sequence is not determined by the Khovanov homology of the links involved. The links ${}^2n^{12}_{1705}$ and ${}^2n^{14}_{65798}$ have the same Khovanov homology, and each as two unknot components (see Figure~\ref{strongerex2fig}).  Yet the spectral sequences collapse on different pages ($E_2$ vs $E_3$).

\begin{figure}[H]
  \begin{center}
    \includegraphics[scale=0.3]{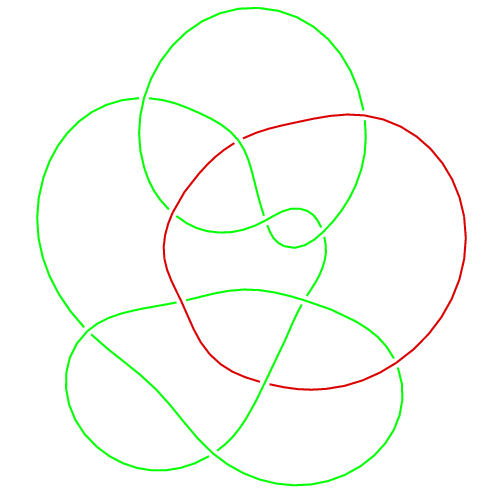}
    \includegraphics[scale=0.3]{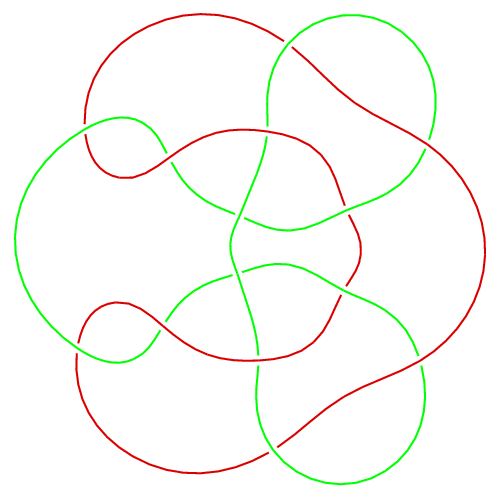}
  \end{center}
  \caption{Links ${}^2n^{12}_{1705}$ (left) and ${}^2n^{14}_{65798}$
    (right). \label{strongerex2fig}}
\end{figure}

\footnotesize % \tiny \scriptsize \footnotesize \small
\begin{center}
\begin{longtable}{p{0.15\textwidth}@{}p{0.03\textwidth}p{0.07\textwidth}p{0.75\textwidth}}
\caption{The link splitting spectral sequence $E_k(L)$ over $R =
  \Z_2(x_1, \dotsc, x_m)$ with weight function $w_c = x_c$ for
  examples in this section.  $E_1(L) = Kh(L)$
  omitted.}\label{strongerexsstab}\\*
\toprule
Link $L$ & $E_k$ & $\rank E_k$ & $P_k(q, t) = \sum_{i,j} (\rank E_k^{ij})t^iq^j$ \\
\midrule
${}^2n^{12}_{1705}$ & $E_2$ & 4 & $q^{-2} + 2 + q^{2}$ \\
${}^2n^{14}_{65798}$ & $E_2$ & 68 & $t^{-5}q^{-10} + t^{-5}q^{-8} + 2t^{-4}q^{-8} + 2t^{-4}q^{-6} + 2t^{-3}q^{-6} + 2t^{-3}q^{-4} + t^{-2}q^{-6} + 4t^{-2}q^{-4} + 3t^{-2}q^{-2} + 2t^{-1}q^{-4} + 5t^{-1}q^{-2} + 3t^{-1} + 3q^{-2} + 6 + 3q^{2} + 3t^{1} + 5t^{1}q^{2} + 2t^{1}q^{4} + 3t^{2}q^{2} + 4t^{2}q^{4} + t^{2}q^{6} + 2t^{3}q^{4} + 2t^{3}q^{6} + 2t^{4}q^{6} + 2t^{4}q^{8} + t^{5}q^{8} + t^{5}q^{10}$ \\
 & $E_3$ & 4 & $q^{-2} + 2 + q^{2}$ \\
\bottomrule
\end{longtable}
\end{center}
\normalsize

We have computed splitting number bounds for all links with 12 or fewer
crossings in the Morwen hyperbolic link tables from
SnapPy \cite{snappy}. Some choices and approximations must be made, which we 
describe before giving the results.

We use two coefficient rings, $\PP = \Z/2(x)$ and $\Q$. For the former, we weight
component $c$ by $w_c = x^c$, and for the latter, we weight component $c$ by the
integer $c$ itself.

Since {\tt knotkit} is not currently able to detect split links, we need
an approximation to the bound coming from the linking number, $b_\text{lk}$.
The link table contains only non-split links, there is no problem for
two-component links. But non-split links with more than two components, such
as the Borromean rings, may have split sublinks. We define $b'_{lk}$ as follows:  If $L$ has two components and is known to
be non-split, we set $b'_{lk}(L) = b_\text{lk}(L)$.  If $L
= K_1\cup K_2$ may be split, then we define
\[ b'_{lk}(L) = \begin{cases}
   |lk(K_1, K_2)| & lk(K_1, K_2)\neq 0 \\
   2 & Kh(L; \Z/2) \not\cong Kh(K_1\amalg K_2; \Z/2) \\
   0 & \text{otherwise} \\
 \end{cases}.
\]
If $L$ has more than two components and is non-split, we define
\[ b'_{lk}(L) = \max\Bigl(\sum_{i < j} b'_{lk}(L_{ij}), 2\Bigr), \]
where $L_{ij}$ is the sublink of $L$ consisting of the $i^\text{th}$
and $j^\text{th}$ components.

Any diagram $D$ for a link $L$ gives an upper bound on the splitting number. 
Number the components of $L$ from $1$ to $m$.
Let $\sigma\in S_m$ be a permutation of the components.  We can
produce a diagram $D'$ for a split link by swapping the
$u(D, \sigma)$ crossings of $D$ where $\sigma(c_\text{upper})
< \sigma(c_\text{lower})$.  Let $u(D)$ be the minimum of
$u(D, \sigma)$ over all $\sigma$, so  $u(D)$ is a upper bound for
$\Sp(L)$.

We computed
\[ b'_{lk}(L), b^\Q(L, w^\Q), b^\PP(L, w^\PP), u(D), \]
for all 5698 links in the Morwen link table with 12 or fewer crossings, where $D$ is the tabulated minimal diagram.  Of those
links, 4770 (83.7\%) have non-trivial lower bounds for $\Sp(L)$ and the lower bound
is known to be tight for 3587 (63\%) links.  Our upper bound is very rough, so the
lower bound is likely to be tight in many more cases.  The bound coming
from the spectral sequence is stronger than the linking number bound for 17 of those
links and equal to it for 2421.  The examples with $b >
b'_{lk}$ are shown Table~\ref{bstrongerextab}.  For those 17 examples,
we verified by hand that $b'_{lk} = b_{lk}$.

We have also tabulated in Table~\ref{spboundcomptab} in the appendix
all links with 10 or fewer crossings for which any of $b'_{lk}, b^\Q$
or $b^\PP$ give a nontrivial lower bound on $\Sp(L)$, .

\footnotesize % \tiny \scriptsize \footnotesize \small
\begin{center}
\begin{longtable}{p{0.15\textwidth}p{0.10\textwidth}p{0.10\textwidth}p{0.10\textwidth}}
\caption{Knots with 12 of fewer crossings for which $b'_{lk}(L) < b(L)$.  $u(D)$ gives the upper bound on $\Sp(L)$.}
\label{bstrongerextab}\\*
\toprule
Link $L$ & $b_\text{lk}(L)$ & $b(L)$ & $\Sp(L)$ \\
\midrule
 ${}^{2}n^{12}_{5538}$  & 1 & 3 & 3 \\
 ${}^{2}n^{12}_{5546}$  & 1 & 3 & 3 \\
 ${}^{2}n^{12}_{5553}$  & 1 & 3 & 3 \\
 ${}^{2}n^{12}_{5559}$  & 1 & 3 & 3 \\
 ${}^{2}n^{12}_{5563}$  & 1 & 3 & 3 \\
 ${}^{2}n^{12}_{5570}$  & 1 & 3 & 3 \\
 ${}^{2}n^{12}_{5600}$  & 1 & 3 & 3 \\
 ${}^{2}a^{11}_{739}$  & 1 & 3 & $3$--$5$ \\
 ${}^{2}a^{12}_{2521}$  & 1 & 3 & $3$--$5$ \\
 ${}^{2}a^{12}_{2552}$  & 1 & 3 & $3$--$5$ \\
 ${}^{2}a^{12}_{2672}$  & 1 & 3 & $3$--$5$ \\
 ${}^{2}n^{12}_{5515}$  & 1 & 3 & $3$--$5$ \\
 ${}^{2}n^{12}_{5516}$  & 1 & 3 & $3$--$5$ \\
 ${}^{2}n^{12}_{5517}$  & 1 & 3 & $3$--$5$ \\
 ${}^{2}n^{12}_{5519}$  & 1 & 3 & $3$--$5$ \\
 ${}^{2}n^{12}_{5522}$  & 1 & 3 & $3$--$5$ \\
 ${}^{3}a^{12}_{2910}$  & 1 & 3 & $3$--$5$ \\
\bottomrule
\end{longtable}
\end{center}
\normalsize

\appendix
\section{Index of notation}
	
\begin{center}
\begin{longtable}{ll}
$R$ & a ring \\
$L$ & a link \\
${}^m[an]^n_k$ & the $k^\text{th}$ $n$-crossing (non-)alternating link in the Morwen \\
  & hyperbolic link table \\
$m = |L|$ & number of components of $L$ \\
$c, d, \dotsc$ & index of components \\
$K_c$ ($1\le c\le m)$ & components of $L$ \\
$A_m$ & the component algebra \\
$\calA$ & Khovanov's $(1 + 1)$-dimensional TQFT \\
$\calA(I, J)$ & Khovanov edge map \\
$\calA(J, I)$ & reverse edge map \\
$Kh(L)$ & Khovanov homology \\
$w$ & component weight function \\
$w_c$ & weight of component $c$ \\
$E_k = E_k(L, w)$ & the link splitting spectral sequence \\
$D$ & a diagram for $L$ \\
$n$ & number of crossings of $D$ \\
$n_+(D)$ & number of positive crossings in $D$ \\
$n_-(D)$ & number of negative crossings in $D$ \\
$i, j, \dotsc$ & general index or index of crossing \\
$I, J, \dotsc$ & (complete) resolution of $D$ \\
$(I, J)$ & edge of the hypercube of resolutions \\
$d_0$ & the Khovanov differential \\
$d_1$ & the deformation of $d_0$ defining $E_k$ \\
$d = d_0 + d_1$ & the total differential for $E_k$ \\
$C = C(D, w)$ & the filtered chain complex which induces $E_k$ \\
$q(x)$ & quantum grading on $C$ \\
$h(x)$ & homological grading on $C$ \\
$\ell(x) = h(x) - q(x)$ & $\ell$ grading \\
$g(x) = (q(x) - m)/2$ & $g$ grading \\
$\mathcal{F}^p$ & filtration on $C$ induced by $g$ inducing $E_k$ \\
$w^i_\text{over}$, $w^i_\text{under}$ & weight of over-, under-strand, respectively, at crossing $i$ \\
$s$ & sign assignment \\
$X$ & CW decomposition of $S^2$ induced by $D$ \\
$h(e, i)$ & height indicator function on endpoints of $1$-cells in $X$ \\
\end{longtable}
\end{center}

\section{Computations}
	\tiny % \tiny \scriptsize \footnotesize \small
\begin{center}
\begin{longtable}{p{0.10\textwidth}@{}p{0.03\textwidth}p{0.07\textwidth}p{0.80\textwidth}}
\caption{Computation of the link splitting spectral sequence for links
  with 8 or fewer crossings from the MT hyperbolic link table with
  non-trivial spectral sequence over $R = \Z_2(x_1, \dotsc, x_m)$ with
  weight function $w_c = x_c$.}
\label{sscomptab}\\*
\toprule
Link $L$ & $E_k$ & $\rank E_k$ & $P_k(q, t) = \sum_{i,j} (\rank E_k^{ij})t^iq^j$ \\
\midrule
${}^2a^{5}_{3}$ & $E_1$ & 16 & $t^{-2}q^{-4} + t^{-2}q^{-2} + t^{-1}q^{-2} + t^{-1} + 2 + 2q^{2} + t^{1}q^{2} + t^{1}q^{4} + 2t^{2}q^{4} + 2t^{2}q^{6} + t^{3}q^{6} + t^{3}q^{8}$ \\
 & $E_2$ & 4 & $t^{-2}q^{-4} + t^{-2}q^{-2} + 1 + q^{2}$ \\
${}^2a^{6}_{4}$ & $E_1$ & 24 & $t^{-4}q^{-10} + t^{-4}q^{-8} + t^{-3}q^{-8} + t^{-3}q^{-6} + 3t^{-2}q^{-6} + 3t^{-2}q^{-4} + 2t^{-1}q^{-4} + 2t^{-1}q^{-2} + 2q^{-2} + 2 + 2t^{1} + 2t^{1}q^{2} + t^{2}q^{2} + t^{2}q^{4}$ \\
 & $E_2$ & 4 & $t^{-4}q^{-10} + t^{-4}q^{-8} + t^{-2}q^{-6} + t^{-2}q^{-4}$ \\
${}^2a^{6}_{5}$ & $E_1$ & 20 & $t^{-6}q^{-16} + t^{-6}q^{-14} + t^{-5}q^{-14} + t^{-5}q^{-12} + 2t^{-4}q^{-12} + 2t^{-4}q^{-10} + 2t^{-3}q^{-10} + 2t^{-3}q^{-8} + 2t^{-2}q^{-8} + 2t^{-2}q^{-6} + t^{-1}q^{-6} + t^{-1}q^{-4} + q^{-4} + q^{-2}$ \\
 & $E_2$ & 4 & $t^{-4}q^{-12} + t^{-4}q^{-10} + t^{-2}q^{-8} + t^{-2}q^{-6}$ \\
${}^2a^{6}_{6}$ & $E_1$ & 12 & $t^{-6}q^{-18} + t^{-6}q^{-16} + t^{-5}q^{-16} + t^{-5}q^{-14} + t^{-4}q^{-14} + t^{-4}q^{-12} + t^{-3}q^{-12} + t^{-3}q^{-10} + t^{-2}q^{-10} + t^{-2}q^{-8} + q^{-6} + q^{-4}$ \\
 & $E_2$ & 4 & $t^{-2}q^{-10} + t^{-2}q^{-8} + q^{-6} + q^{-4}$ \\
${}^3a^{6}_{7}$ & $E_1$ & 32 & $t^{-3}q^{-7} + t^{-3}q^{-5} + 3t^{-2}q^{-5} + 3t^{-2}q^{-3} + 2t^{-1}q^{-3} + 2t^{-1}q^{-1} + 4q^{-1} + 4q^{1} + 2t^{1}q^{1} + 2t^{1}q^{3} + 3t^{2}q^{3} + 3t^{2}q^{5} + t^{3}q^{5} + t^{3}q^{7}$ \\
 & $E_2$ & 8 & $t^{-2}q^{-5} + t^{-2}q^{-3} + 2q^{-1} + 2q^{1} + t^{2}q^{3} + t^{2}q^{5}$ \\
${}^3a^{6}_{8}$ & $E_1$ & 24 & $t^{-6}q^{-15} + t^{-6}q^{-13} + t^{-5}q^{-13} + t^{-5}q^{-11} + 3t^{-4}q^{-11} + 3t^{-4}q^{-9} + t^{-3}q^{-9} + t^{-3}q^{-7} + 3t^{-2}q^{-7} + 3t^{-2}q^{-5} + 2t^{-1}q^{-5} + 2t^{-1}q^{-3} + q^{-3} + q^{-1}$ \\
 & $E_2$ & 8 & $t^{-6}q^{-15} + t^{-6}q^{-13} + 2t^{-4}q^{-11} + 2t^{-4}q^{-9} + t^{-2}q^{-7} + t^{-2}q^{-5}$ \\
${}^3n^{6}_{1}$ & $E_1$ & 12 & $2q^{-1} + 3q^{1} + q^{3} + t^{1}q^{1} + t^{1}q^{3} + t^{2}q^{3} + t^{2}q^{5} + t^{4}q^{7} + t^{4}q^{9}$ \\
 & $E_2$ & 8 & $q^{-1} + 2q^{1} + q^{3} + t^{2}q^{3} + t^{2}q^{5} + t^{4}q^{7} + t^{4}q^{9}$ \\
${}^2a^{7}_{8}$ & $E_1$ & 48 & $t^{-3}q^{-6} + t^{-3}q^{-4} + 3t^{-2}q^{-4} + 3t^{-2}q^{-2} + 3t^{-1}q^{-2} + 3t^{-1} + 5 + 5q^{2} + 4t^{1}q^{2} + 4t^{1}q^{4} + 4t^{2}q^{4} + 4t^{2}q^{6} + 3t^{3}q^{6} + 3t^{3}q^{8} + t^{4}q^{8} + t^{4}q^{10}$ \\
 & $E_2$ & 4 & $t^{-2}q^{-4} + t^{-2}q^{-2} + 1 + q^{2}$ \\
${}^2a^{7}_{9}$ & $E_1$ & 40 & $t^{-7}q^{-18} + t^{-7}q^{-16} + 2t^{-6}q^{-16} + 2t^{-6}q^{-14} + 3t^{-5}q^{-14} + 3t^{-5}q^{-12} + 4t^{-4}q^{-12} + 4t^{-4}q^{-10} + 3t^{-3}q^{-10} + 3t^{-3}q^{-8} + 4t^{-2}q^{-8} + 4t^{-2}q^{-6} + 2t^{-1}q^{-6} + 2t^{-1}q^{-4} + q^{-4} + q^{-2}$ \\
 & $E_2$ & 12 & $t^{-7}q^{-18} + t^{-7}q^{-16} + t^{-6}q^{-16} + t^{-6}q^{-14} + t^{-5}q^{-14} + t^{-5}q^{-12} + 2t^{-4}q^{-12} + 2t^{-4}q^{-10} + t^{-2}q^{-8} + t^{-2}q^{-6}$ \\
${}^2a^{7}_{10}$ & $E_1$ & 32 & $t^{-2}q^{-2} + t^{-2} + t^{-1} + t^{-1}q^{2} + 3q^{2} + 3q^{4} + 2t^{1}q^{4} + 2t^{1}q^{6} + 3t^{2}q^{6} + 3t^{2}q^{8} + 3t^{3}q^{8} + 3t^{3}q^{10} + 2t^{4}q^{10} + 2t^{4}q^{12} + t^{5}q^{12} + t^{5}q^{14}$ \\
 & $E_2$ & 12 & $t^{-2}q^{-2} + t^{-2} + 2q^{2} + 2q^{4} + t^{1}q^{4} + t^{1}q^{6} + t^{2}q^{6} + t^{2}q^{8} + t^{3}q^{8} + t^{3}q^{10}$ \\
${}^2a^{7}_{11}$ & $E_1$ & 32 & $t^{-2}q^{-4} + t^{-2}q^{-2} + t^{-1}q^{-2} + t^{-1} + 3 + 3q^{2} + 3t^{1}q^{2} + 3t^{1}q^{4} + 3t^{2}q^{4} + 3t^{2}q^{6} + 2t^{3}q^{6} + 2t^{3}q^{8} + 2t^{4}q^{8} + 2t^{4}q^{10} + t^{5}q^{10} + t^{5}q^{12}$ \\
 & $E_2$ & 4 & $t^{-2}q^{-4} + t^{-2}q^{-2} + 1 + q^{2}$ \\
${}^2a^{7}_{12}$ & $E_1$ & 36 & $t^{-5}q^{-12} + t^{-5}q^{-10} + 2t^{-4}q^{-10} + 2t^{-4}q^{-8} + 2t^{-3}q^{-8} + 2t^{-3}q^{-6} + 4t^{-2}q^{-6} + 4t^{-2}q^{-4} + 3t^{-1}q^{-4} + 3t^{-1}q^{-2} + 3q^{-2} + 3 + 2t^{1} + 2t^{1}q^{2} + t^{2}q^{2} + t^{2}q^{4}$ \\
 & $E_2$ & 4 & $t^{-2}q^{-6} + t^{-2}q^{-4} + q^{-2} + 1$ \\
${}^2a^{7}_{13}$ & $E_1$ & 28 & $t^{-2}q^{-2} + t^{-2} + t^{-1} + t^{-1}q^{2} + 2q^{2} + 2q^{4} + 2t^{1}q^{4} + 2t^{1}q^{6} + 3t^{2}q^{6} + 3t^{2}q^{8} + 2t^{3}q^{8} + 2t^{3}q^{10} + 2t^{4}q^{10} + 2t^{4}q^{12} + t^{5}q^{12} + t^{5}q^{14}$ \\
 & $E_2$ & 4 & $t^{-2}q^{-2} + t^{-2} + q^{2} + q^{4}$ \\
${}^3a^{7}_{14}$ & $E_1$ & 40 & $t^{-4}q^{-9} + t^{-4}q^{-7} + t^{-3}q^{-7} + t^{-3}q^{-5} + 4t^{-2}q^{-5} + 4t^{-2}q^{-3} + 3t^{-1}q^{-3} + 3t^{-1}q^{-1} + 4q^{-1} + 4q^{1} + 3t^{1}q^{1} + 3t^{1}q^{3} + 3t^{2}q^{3} + 3t^{2}q^{5} + t^{3}q^{5} + t^{3}q^{7}$ \\
 & $E_2$ & 8 & $t^{-4}q^{-9} + t^{-4}q^{-7} + 2t^{-2}q^{-5} + 2t^{-2}q^{-3} + q^{-1} + q^{1}$ \\
${}^2n^{7}_{2}$ & $E_1$ & 16 & $t^{-5}q^{-12} + t^{-5}q^{-10} + t^{-4}q^{-10} + t^{-4}q^{-8} + t^{-3}q^{-8} + t^{-3}q^{-6} + 2t^{-2}q^{-6} + 2t^{-2}q^{-4} + t^{-1}q^{-4} + t^{-1}q^{-2} + 2q^{-2} + 2$ \\
 & $E_2$ & 12 & $t^{-5}q^{-12} + t^{-5}q^{-10} + t^{-4}q^{-10} + t^{-4}q^{-8} + t^{-3}q^{-8} + t^{-3}q^{-6} + 2t^{-2}q^{-6} + 2t^{-2}q^{-4} + q^{-2} + 1$ \\
${}^2a^{8}_{19}$ & $E_1$ & 80 & $t^{-5}q^{-12} + t^{-5}q^{-10} + 3t^{-4}q^{-10} + 3t^{-4}q^{-8} + 5t^{-3}q^{-8} + 5t^{-3}q^{-6} + 7t^{-2}q^{-6} + 7t^{-2}q^{-4} + 7t^{-1}q^{-4} + 7t^{-1}q^{-2} + 7q^{-2} + 7 + 5t^{1} + 5t^{1}q^{2} + 4t^{2}q^{2} + 4t^{2}q^{4} + t^{3}q^{4} + t^{3}q^{6}$ \\
 & $E_2$ & 4 & $q^{-2} + 1 + t^{2}q^{2} + t^{2}q^{4}$ \\
${}^2a^{8}_{20}$ & $E_1$ & 64 & $t^{-4}q^{-8} + t^{-4}q^{-6} + 2t^{-3}q^{-6} + 2t^{-3}q^{-4} + 4t^{-2}q^{-4} + 4t^{-2}q^{-2} + 5t^{-1}q^{-2} + 5t^{-1} + 6 + 6q^{2} + 5t^{1}q^{2} + 5t^{1}q^{4} + 5t^{2}q^{4} + 5t^{2}q^{6} + 3t^{3}q^{6} + 3t^{3}q^{8} + t^{4}q^{8} + t^{4}q^{10}$ \\
 & $E_2$ & 20 & $t^{-4}q^{-8} + t^{-4}q^{-6} + t^{-3}q^{-6} + t^{-3}q^{-4} + 2t^{-2}q^{-4} + 2t^{-2}q^{-2} + 2t^{-1}q^{-2} + 2t^{-1} + 2 + 2q^{2} + t^{1}q^{2} + t^{1}q^{4} + t^{2}q^{4} + t^{2}q^{6}$ \\
${}^2a^{8}_{21}$ & $E_1$ & 56 & $t^{-6}q^{-14} + t^{-6}q^{-12} + 2t^{-5}q^{-12} + 2t^{-5}q^{-10} + 4t^{-4}q^{-10} + 4t^{-4}q^{-8} + 4t^{-3}q^{-8} + 4t^{-3}q^{-6} + 5t^{-2}q^{-6} + 5t^{-2}q^{-4} + 5t^{-1}q^{-4} + 5t^{-1}q^{-2} + 4q^{-2} + 4 + 2t^{1} + 2t^{1}q^{2} + t^{2}q^{2} + t^{2}q^{4}$ \\
 & $E_2$ & 20 & $t^{-6}q^{-14} + t^{-6}q^{-12} + t^{-5}q^{-12} + t^{-5}q^{-10} + 2t^{-4}q^{-10} + 2t^{-4}q^{-8} + 2t^{-3}q^{-8} + 2t^{-3}q^{-6} + 2t^{-2}q^{-6} + 2t^{-2}q^{-4} + t^{-1}q^{-4} + t^{-1}q^{-2} + q^{-2} + 1$ \\
${}^2a^{8}_{22}$ & $E_1$ & 64 & $t^{-5}q^{-12} + t^{-5}q^{-10} + 2t^{-4}q^{-10} + 2t^{-4}q^{-8} + 4t^{-3}q^{-8} + 4t^{-3}q^{-6} + 6t^{-2}q^{-6} + 6t^{-2}q^{-4} + 5t^{-1}q^{-4} + 5t^{-1}q^{-2} + 6q^{-2} + 6 + 4t^{1} + 4t^{1}q^{2} + 3t^{2}q^{2} + 3t^{2}q^{4} + t^{3}q^{4} + t^{3}q^{6}$ \\
 & $E_2$ & 12 & $t^{-5}q^{-12} + t^{-5}q^{-10} + t^{-4}q^{-10} + t^{-4}q^{-8} + t^{-3}q^{-8} + t^{-3}q^{-6} + 2t^{-2}q^{-6} + 2t^{-2}q^{-4} + q^{-2} + 1$ \\
${}^2a^{8}_{23}$ & $E_1$ & 56 & $t^{-4}q^{-8} + t^{-4}q^{-6} + t^{-3}q^{-6} + t^{-3}q^{-4} + 4t^{-2}q^{-4} + 4t^{-2}q^{-2} + 4t^{-1}q^{-2} + 4t^{-1} + 5 + 5q^{2} + 5t^{1}q^{2} + 5t^{1}q^{4} + 4t^{2}q^{4} + 4t^{2}q^{6} + 3t^{3}q^{6} + 3t^{3}q^{8} + t^{4}q^{8} + t^{4}q^{10}$ \\
 & $E_2$ & 12 & $t^{-4}q^{-8} + t^{-4}q^{-6} + 2t^{-2}q^{-4} + 2t^{-2}q^{-2} + t^{-1}q^{-2} + t^{-1} + 1 + q^{2} + t^{1}q^{2} + t^{1}q^{4}$ \\
${}^2a^{8}_{24}$ & $E_1$ & 40 & $t^{-4}q^{-10} + t^{-4}q^{-8} + t^{-3}q^{-8} + t^{-3}q^{-6} + 3t^{-2}q^{-6} + 3t^{-2}q^{-4} + 3t^{-1}q^{-4} + 3t^{-1}q^{-2} + 4q^{-2} + 4 + 3t^{1} + 3t^{1}q^{2} + 2t^{2}q^{2} + 2t^{2}q^{4} + 2t^{3}q^{4} + 2t^{3}q^{6} + t^{4}q^{6} + t^{4}q^{8}$ \\
 & $E_2$ & 4 & $t^{-4}q^{-10} + t^{-4}q^{-8} + t^{-2}q^{-6} + t^{-2}q^{-4}$ \\
${}^2a^{8}_{25}$ & $E_1$ & 72 & $t^{-8}q^{-20} + t^{-8}q^{-18} + 3t^{-7}q^{-18} + 3t^{-7}q^{-16} + 4t^{-6}q^{-16} + 4t^{-6}q^{-14} + 6t^{-5}q^{-14} + 6t^{-5}q^{-12} + 7t^{-4}q^{-12} + 7t^{-4}q^{-10} + 5t^{-3}q^{-10} + 5t^{-3}q^{-8} + 6t^{-2}q^{-8} + 6t^{-2}q^{-6} + 3t^{-1}q^{-6} + 3t^{-1}q^{-4} + q^{-4} + q^{-2}$ \\
 & $E_2$ & 12 & $t^{-7}q^{-18} + t^{-7}q^{-16} + t^{-6}q^{-16} + t^{-6}q^{-14} + t^{-5}q^{-14} + t^{-5}q^{-12} + 2t^{-4}q^{-12} + 2t^{-4}q^{-10} + t^{-2}q^{-8} + t^{-2}q^{-6}$ \\
${}^2a^{8}_{26}$ & $E_1$ & 60 & $t^{-4}q^{-8} + t^{-4}q^{-6} + 2t^{-3}q^{-6} + 2t^{-3}q^{-4} + 4t^{-2}q^{-4} + 4t^{-2}q^{-2} + 4t^{-1}q^{-2} + 4t^{-1} + 6 + 6q^{2} + 5t^{1}q^{2} + 5t^{1}q^{4} + 4t^{2}q^{4} + 4t^{2}q^{6} + 3t^{3}q^{6} + 3t^{3}q^{8} + t^{4}q^{8} + t^{4}q^{10}$ \\
 & $E_2$ & 4 & $t^{-4}q^{-8} + t^{-4}q^{-6} + t^{-2}q^{-4} + t^{-2}q^{-2}$ \\
${}^2a^{8}_{27}$ & $E_1$ & 68 & $t^{-5}q^{-12} + t^{-5}q^{-10} + 3t^{-4}q^{-10} + 3t^{-4}q^{-8} + 4t^{-3}q^{-8} + 4t^{-3}q^{-6} + 6t^{-2}q^{-6} + 6t^{-2}q^{-4} + 6t^{-1}q^{-4} + 6t^{-1}q^{-2} + 6q^{-2} + 6 + 4t^{1} + 4t^{1}q^{2} + 3t^{2}q^{2} + 3t^{2}q^{4} + t^{3}q^{4} + t^{3}q^{6}$ \\
 & $E_2$ & 4 & $t^{-2}q^{-6} + t^{-2}q^{-4} + q^{-2} + 1$ \\
${}^2a^{8}_{28}$ & $E_1$ & 52 & $t^{-8}q^{-20} + t^{-8}q^{-18} + 2t^{-7}q^{-18} + 2t^{-7}q^{-16} + 3t^{-6}q^{-16} + 3t^{-6}q^{-14} + 4t^{-5}q^{-14} + 4t^{-5}q^{-12} + 5t^{-4}q^{-12} + 5t^{-4}q^{-10} + 4t^{-3}q^{-10} + 4t^{-3}q^{-8} + 4t^{-2}q^{-8} + 4t^{-2}q^{-6} + 2t^{-1}q^{-6} + 2t^{-1}q^{-4} + q^{-4} + q^{-2}$ \\
 & $E_2$ & 4 & $t^{-4}q^{-12} + t^{-4}q^{-10} + t^{-2}q^{-8} + t^{-2}q^{-6}$ \\
${}^2a^{8}_{29}$ & $E_1$ & 44 & $t^{-8}q^{-22} + t^{-8}q^{-20} + 2t^{-7}q^{-20} + 2t^{-7}q^{-18} + 3t^{-6}q^{-18} + 3t^{-6}q^{-16} + 4t^{-5}q^{-16} + 4t^{-5}q^{-14} + 4t^{-4}q^{-14} + 4t^{-4}q^{-12} + 3t^{-3}q^{-12} + 3t^{-3}q^{-10} + 3t^{-2}q^{-10} + 3t^{-2}q^{-8} + t^{-1}q^{-8} + t^{-1}q^{-6} + q^{-6} + q^{-4}$ \\
 & $E_2$ & 4 & $t^{-2}q^{-10} + t^{-2}q^{-8} + q^{-6} + q^{-4}$ \\
${}^2a^{8}_{30}$ & $E_1$ & 32 & $t^{-8}q^{-22} + t^{-8}q^{-20} + t^{-7}q^{-20} + t^{-7}q^{-18} + 2t^{-6}q^{-18} + 2t^{-6}q^{-16} + 3t^{-5}q^{-16} + 3t^{-5}q^{-14} + 3t^{-4}q^{-14} + 3t^{-4}q^{-12} + 2t^{-3}q^{-12} + 2t^{-3}q^{-10} + 2t^{-2}q^{-10} + 2t^{-2}q^{-8} + t^{-1}q^{-8} + t^{-1}q^{-6} + q^{-6} + q^{-4}$ \\
 & $E_2$ & 4 & $t^{-4}q^{-14} + t^{-4}q^{-12} + t^{-2}q^{-10} + t^{-2}q^{-8}$ \\
${}^2a^{8}_{31}$ & $E_1$ & 48 & $t^{-8}q^{-20} + t^{-8}q^{-18} + t^{-7}q^{-18} + t^{-7}q^{-16} + 3t^{-6}q^{-16} + 3t^{-6}q^{-14} + 4t^{-5}q^{-14} + 4t^{-5}q^{-12} + 4t^{-4}q^{-12} + 4t^{-4}q^{-10} + 4t^{-3}q^{-10} + 4t^{-3}q^{-8} + 4t^{-2}q^{-8} + 4t^{-2}q^{-6} + 2t^{-1}q^{-6} + 2t^{-1}q^{-4} + q^{-4} + q^{-2}$ \\
 & $E_2$ & 4 & $t^{-6}q^{-16} + t^{-6}q^{-14} + t^{-4}q^{-12} + t^{-4}q^{-10}$ \\
${}^2a^{8}_{32}$ & $E_1$ & 16 & $t^{-8}q^{-24} + t^{-8}q^{-22} + t^{-7}q^{-22} + t^{-7}q^{-20} + t^{-6}q^{-20} + t^{-6}q^{-18} + t^{-5}q^{-18} + t^{-5}q^{-16} + t^{-4}q^{-16} + t^{-4}q^{-14} + t^{-3}q^{-14} + t^{-3}q^{-12} + t^{-2}q^{-12} + t^{-2}q^{-10} + q^{-8} + q^{-6}$ \\
 & $E_2$ & 4 & $t^{-2}q^{-12} + t^{-2}q^{-10} + q^{-8} + q^{-6}$ \\
${}^3a^{8}_{33}$ & $E_1$ & 56 & $t^{-6}q^{-15} + t^{-6}q^{-13} + t^{-5}q^{-13} + t^{-5}q^{-11} + 4t^{-4}q^{-11} + 4t^{-4}q^{-9} + 4t^{-3}q^{-9} + 4t^{-3}q^{-7} + 6t^{-2}q^{-7} + 6t^{-2}q^{-5} + 4t^{-1}q^{-5} + 4t^{-1}q^{-3} + 4q^{-3} + 4q^{-1} + 3t^{1}q^{-1} + 3t^{1}q^{1} + t^{2}q^{1} + t^{2}q^{3}$ \\
 & $E_2$ & 8 & $t^{-6}q^{-15} + t^{-6}q^{-13} + 2t^{-4}q^{-11} + 2t^{-4}q^{-9} + t^{-2}q^{-7} + t^{-2}q^{-5}$ \\
${}^3a^{8}_{34}$ & $E_1$ & 64 & $t^{-3}q^{-5} + t^{-3}q^{-3} + 3t^{-2}q^{-3} + 3t^{-2}q^{-1} + 3t^{-1}q^{-1} + 3t^{-1}q^{1} + 6q^{1} + 6q^{3} + 5t^{1}q^{3} + 5t^{1}q^{5} + 6t^{2}q^{5} + 6t^{2}q^{7} + 4t^{3}q^{7} + 4t^{3}q^{9} + 3t^{4}q^{9} + 3t^{4}q^{11} + t^{5}q^{11} + t^{5}q^{13}$ \\
 & $E_2$ & 8 & $t^{-2}q^{-3} + t^{-2}q^{-1} + 2q^{1} + 2q^{3} + t^{2}q^{5} + t^{2}q^{7}$ \\
${}^3a^{8}_{35}$ & $E_1$ & 56 & $t^{-8}q^{-21} + t^{-8}q^{-19} + 2t^{-7}q^{-19} + 2t^{-7}q^{-17} + 4t^{-6}q^{-17} + 4t^{-6}q^{-15} + 4t^{-5}q^{-15} + 4t^{-5}q^{-13} + 6t^{-4}q^{-13} + 6t^{-4}q^{-11} + 4t^{-3}q^{-11} + 4t^{-3}q^{-9} + 4t^{-2}q^{-9} + 4t^{-2}q^{-7} + 2t^{-1}q^{-7} + 2t^{-1}q^{-5} + q^{-5} + q^{-3}$ \\
 & $E_2$ & 8 & $t^{-6}q^{-17} + t^{-6}q^{-15} + 2t^{-4}q^{-13} + 2t^{-4}q^{-11} + t^{-2}q^{-9} + t^{-2}q^{-7}$ \\
${}^3a^{8}_{36}$ & $E_1$ & 40 & $t^{-2}q^{-1} + t^{-2}q^{1} + t^{-1}q^{1} + t^{-1}q^{3} + 3q^{3} + 3q^{5} + 2t^{1}q^{5} + 2t^{1}q^{7} + 4t^{2}q^{7} + 4t^{2}q^{9} + 3t^{3}q^{9} + 3t^{3}q^{11} + 3t^{4}q^{11} + 3t^{4}q^{13} + 2t^{5}q^{13} + 2t^{5}q^{15} + t^{6}q^{15} + t^{6}q^{17}$ \\
 & $E_2$ & 8 & $t^{-2}q^{-1} + t^{-2}q^{1} + 2q^{3} + 2q^{5} + t^{2}q^{7} + t^{2}q^{9}$ \\
${}^3a^{8}_{37}$ & $E_1$ & 72 & $t^{-4}q^{-9} + t^{-4}q^{-7} + 3t^{-3}q^{-7} + 3t^{-3}q^{-5} + 5t^{-2}q^{-5} + 5t^{-2}q^{-3} + 5t^{-1}q^{-3} + 5t^{-1}q^{-1} + 8q^{-1} + 8q^{1} + 5t^{1}q^{1} + 5t^{1}q^{3} + 5t^{2}q^{3} + 5t^{2}q^{5} + 3t^{3}q^{5} + 3t^{3}q^{7} + t^{4}q^{7} + t^{4}q^{9}$ \\
 & $E_2$ & 8 & $t^{-2}q^{-5} + t^{-2}q^{-3} + 2q^{-1} + 2q^{1} + t^{2}q^{3} + t^{2}q^{5}$ \\
${}^3a^{8}_{38}$ & $E_1$ & 64 & $t^{-4}q^{-9} + t^{-4}q^{-7} + 2t^{-3}q^{-7} + 2t^{-3}q^{-5} + 5t^{-2}q^{-5} + 5t^{-2}q^{-3} + 5t^{-1}q^{-3} + 5t^{-1}q^{-1} + 6q^{-1} + 6q^{1} + 5t^{1}q^{1} + 5t^{1}q^{3} + 5t^{2}q^{3} + 5t^{2}q^{5} + 2t^{3}q^{5} + 2t^{3}q^{7} + t^{4}q^{7} + t^{4}q^{9}$ \\
 & $E_2$ & 8 & $t^{-2}q^{-5} + t^{-2}q^{-3} + 2q^{-1} + 2q^{1} + t^{2}q^{3} + t^{2}q^{5}$ \\
${}^4a^{8}_{39}$ & $E_1$ & 64 & $t^{-8}q^{-20} + t^{-8}q^{-18} + t^{-7}q^{-18} + t^{-7}q^{-16} + 5t^{-6}q^{-16} + 5t^{-6}q^{-14} + 4t^{-5}q^{-14} + 4t^{-5}q^{-12} + 7t^{-4}q^{-12} + 7t^{-4}q^{-10} + 4t^{-3}q^{-10} + 4t^{-3}q^{-8} + 6t^{-2}q^{-8} + 6t^{-2}q^{-6} + 3t^{-1}q^{-6} + 3t^{-1}q^{-4} + q^{-4} + q^{-2}$ \\
 & $E_2$ & 16 & $t^{-8}q^{-20} + t^{-8}q^{-18} + 3t^{-6}q^{-16} + 3t^{-6}q^{-14} + 3t^{-4}q^{-12} + 3t^{-4}q^{-10} + t^{-2}q^{-8} + t^{-2}q^{-6}$ \\
${}^2n^{8}_{4}$ & $E_1$ & 24 & $2t^{-4}q^{-12} + 2t^{-4}q^{-10} + 2t^{-3}q^{-10} + 2t^{-3}q^{-8} + 2t^{-2}q^{-8} + 2t^{-2}q^{-6} + 2t^{-1}q^{-6} + 2t^{-1}q^{-4} + 2q^{-4} + 2q^{-2} + t^{1}q^{-2} + t^{1} + t^{2} + t^{2}q^{2}$ \\
 & $E_2$ & 20 & $t^{-4}q^{-12} + t^{-4}q^{-10} + t^{-3}q^{-10} + t^{-3}q^{-8} + 2t^{-2}q^{-8} + 2t^{-2}q^{-6} + 2t^{-1}q^{-6} + 2t^{-1}q^{-4} + 2q^{-4} + 2q^{-2} + t^{1}q^{-2} + t^{1} + t^{2} + t^{2}q^{2}$ \\
${}^3n^{8}_{6}$ & $E_1$ & 16 & $t^{-6}q^{-19} + 2t^{-6}q^{-17} + t^{-6}q^{-15} + t^{-5}q^{-17} + t^{-5}q^{-15} + t^{-4}q^{-15} + 2t^{-4}q^{-13} + t^{-4}q^{-11} + t^{-3}q^{-13} + t^{-3}q^{-11} + t^{-2}q^{-11} + t^{-2}q^{-9} + q^{-7} + q^{-5}$ \\
 & $E_2$ & 8 & $t^{-6}q^{-17} + t^{-6}q^{-15} + t^{-4}q^{-13} + t^{-4}q^{-11} + t^{-2}q^{-11} + t^{-2}q^{-9} + q^{-7} + q^{-5}$ \\
${}^3n^{8}_{7}$ & $E_1$ & 24 & $t^{-6}q^{-15} + t^{-6}q^{-13} + t^{-5}q^{-13} + t^{-5}q^{-11} + 2t^{-4}q^{-11} + 2t^{-4}q^{-9} + 2t^{-3}q^{-9} + 2t^{-3}q^{-7} + 3t^{-2}q^{-7} + 3t^{-2}q^{-5} + t^{-1}q^{-5} + t^{-1}q^{-3} + 2q^{-3} + 2q^{-1}$ \\
 & $E_2$ & 8 & $t^{-4}q^{-11} + t^{-4}q^{-9} + 2t^{-2}q^{-7} + 2t^{-2}q^{-5} + q^{-3} + q^{-1}$ \\
${}^3n^{8}_{8}$ & $E_1$ & 32 & $t^{-6}q^{-15} + t^{-6}q^{-13} + 2t^{-5}q^{-13} + 2t^{-5}q^{-11} + 3t^{-4}q^{-11} + 3t^{-4}q^{-9} + 2t^{-3}q^{-9} + 2t^{-3}q^{-7} + 4t^{-2}q^{-7} + 4t^{-2}q^{-5} + 2t^{-1}q^{-5} + 2t^{-1}q^{-3} + 2q^{-3} + 2q^{-1}$ \\
 & $E_2$ & 8 & $t^{-4}q^{-11} + t^{-4}q^{-9} + 2t^{-2}q^{-7} + 2t^{-2}q^{-5} + q^{-3} + q^{-1}$ \\
${}^3n^{8}_{9}$ & $E_1$ & 20 & $t^{-8}q^{-19} + t^{-8}q^{-17} + t^{-6}q^{-15} + t^{-6}q^{-13} + t^{-5}q^{-13} + t^{-5}q^{-11} + 2t^{-4}q^{-13} + 3t^{-4}q^{-11} + t^{-4}q^{-9} + t^{-3}q^{-11} + 2t^{-3}q^{-9} + t^{-3}q^{-7} + t^{-2}q^{-9} + t^{-2}q^{-7} + q^{-5} + q^{-3}$ \\
 & $E_2$ & 16 & $t^{-8}q^{-19} + t^{-8}q^{-17} + t^{-6}q^{-15} + t^{-6}q^{-13} + t^{-5}q^{-13} + t^{-5}q^{-11} + t^{-4}q^{-13} + 2t^{-4}q^{-11} + t^{-4}q^{-9} + t^{-3}q^{-9} + t^{-3}q^{-7} + t^{-2}q^{-9} + t^{-2}q^{-7} + q^{-5} + q^{-3}$ \\
 & $E_3$ & 8 & $t^{-8}q^{-19} + t^{-8}q^{-17} + t^{-6}q^{-15} + t^{-6}q^{-13} + t^{-4}q^{-13} + 2t^{-4}q^{-11} + t^{-4}q^{-9}$ \\
${}^4n^{8}_{10}$ & $E_1$ & 36 & $3 + 4q^{2} + q^{4} + 3t^{1}q^{2} + 3t^{1}q^{4} + 4t^{2}q^{4} + 4t^{2}q^{6} + t^{3}q^{6} + t^{3}q^{8} + 4t^{4}q^{8} + 4t^{4}q^{10} + t^{5}q^{10} + t^{5}q^{12} + t^{6}q^{12} + t^{6}q^{14}$ \\
 & $E_2$ & 16 & $1 + 2q^{2} + q^{4} + 2t^{2}q^{4} + 2t^{2}q^{6} + 3t^{4}q^{8} + 3t^{4}q^{10} + t^{6}q^{12} + t^{6}q^{14}$ \\
${}^4n^{8}_{11}$ & $E_1$ & 24 & $t^{-4}q^{-8} + t^{-4}q^{-6} + t^{-2}q^{-4} + t^{-2}q^{-2} + t^{-1}q^{-2} + t^{-1} + 3q^{-2} + 6 + 3q^{2} + t^{1} + t^{1}q^{2} + t^{2}q^{2} + t^{2}q^{4} + t^{4}q^{6} + t^{4}q^{8}$ \\
 & $E_2$ & 16 & $t^{-4}q^{-8} + t^{-4}q^{-6} + t^{-2}q^{-4} + t^{-2}q^{-2} + 2q^{-2} + 4 + 2q^{2} + t^{2}q^{2} + t^{2}q^{4} + t^{4}q^{6} + t^{4}q^{8}$ \\
\bottomrule
\end{longtable}
\end{center}
\normalsize

\footnotesize % \tiny \scriptsize \footnotesize \small
\begin{center}
\begin{longtable}{llllllllll}
\caption{Computations of $\Sp(L)$ for links with 10 or fewer crossings.  In these examples, $b(L) \le b'_{lk}(L)$; ${}^*$ indicates cases where $b'_{lk}(L) = b(L)$.}\label{spboundcomptab}\\*
\toprule
link $L$ & $\Sp(L)$ & link $L$ & $\Sp(L)$ & link $L$ & $\Sp(L)$ & link $L$ & $\Sp(L)$ & link $L$ & $\Sp(L)$ \\
\midrule
 ${}^{2}a^{5}_{3}$  & $2^*$ & ${}^{2}a^{6}_{4}$  & $2^*$ & ${}^{2}a^{6}_{5}$  & $3$ & ${}^{2}a^{6}_{6}$  & $3^*$ & ${}^{3}a^{6}_{7}$  & $2^*$ \\
 ${}^{3}a^{6}_{8}$  & $3$ & ${}^{3}n^{6}_{9}$  & $3$ & ${}^{2}a^{7}_{8}$  & $2^*$ & ${}^{2}a^{7}_{9}$  & $2^*$ & ${}^{2}a^{7}_{10}$  & $2^*$ \\
 ${}^{2}a^{7}_{11}$  & $2^*$ & ${}^{3}a^{7}_{14}$  & $3$ & ${}^{2}n^{7}_{15}$  & $2$ & ${}^{2}n^{7}_{16}$  & $2^*$ & ${}^{2}a^{8}_{19}$  & $2^*$ \\
 ${}^{2}a^{8}_{20}$  & $2^*$ & ${}^{2}a^{8}_{21}$  & $2^*$ & ${}^{2}a^{8}_{22}$  & $2^*$ & ${}^{2}a^{8}_{23}$  & $2^*$ & ${}^{2}a^{8}_{24}$  & $2^*$ \\
 ${}^{2}a^{8}_{25}$  & $2^*$ & ${}^{2}a^{8}_{28}$  & $3$ & ${}^{2}a^{8}_{29}$  & $3^*$ & ${}^{2}a^{8}_{30}$  & $4$ & ${}^{2}a^{8}_{31}$  & $4$ \\
 ${}^{2}a^{8}_{32}$  & $4$ & ${}^{3}a^{8}_{33}$  & $3$ & ${}^{3}a^{8}_{35}$  & $4$ & ${}^{3}a^{8}_{36}$  & $4$ & ${}^{3}a^{8}_{37}$  & $2$--$4^*$ \\
 ${}^{3}a^{8}_{38}$  & $4$ & ${}^{4}a^{8}_{39}$  & $4$ & ${}^{2}n^{8}_{43}$  & $2^*$ & ${}^{2}n^{8}_{44}$  & $2$ & ${}^{3}n^{8}_{45}$  & $4$ \\
 ${}^{3}n^{8}_{46}$  & $4$ & ${}^{3}n^{8}_{47}$  & $2^*$ & ${}^{3}n^{8}_{48}$  & $4$ & ${}^{4}n^{8}_{49}$  & $4$ & ${}^{4}n^{8}_{50}$  & $4$ \\
 ${}^{2}a^{9}_{42}$  & $2^*$ & ${}^{2}a^{9}_{43}$  & $2^*$ & ${}^{2}a^{9}_{44}$  & $2^*$ & ${}^{2}a^{9}_{45}$  & $2^*$ & ${}^{2}a^{9}_{46}$  & $2^*$ \\
 ${}^{2}a^{9}_{47}$  & $2^*$ & ${}^{2}a^{9}_{48}$  & $2^*$ & ${}^{2}a^{9}_{49}$  & $2^*$ & ${}^{2}a^{9}_{50}$  & $2^*$ & ${}^{2}a^{9}_{51}$  & $2^*$ \\
 ${}^{2}a^{9}_{52}$  & $2^*$ & ${}^{2}a^{9}_{53}$  & $2^*$ & ${}^{2}a^{9}_{54}$  & $2^*$ & ${}^{2}a^{9}_{55}$  & $2^*$ & ${}^{2}a^{9}_{56}$  & $2^*$ \\
 ${}^{2}a^{9}_{57}$  & $2^*$ & ${}^{2}a^{9}_{58}$  & $2^*$ & ${}^{2}a^{9}_{59}$  & $2^*$ & ${}^{2}a^{9}_{60}$  & $2^*$ & ${}^{2}a^{9}_{64}$  & $3$ \\
 ${}^{2}a^{9}_{69}$  & $3^*$ & ${}^{2}a^{9}_{73}$  & $3^*$ & ${}^{2}a^{9}_{74}$  & $3^*$ & ${}^{2}a^{9}_{75}$  & $2$--$4^*$ & ${}^{2}a^{9}_{76}$  & $2$--$4^*$ \\
 ${}^{2}a^{9}_{77}$  & $2$--$4^*$ & ${}^{2}a^{9}_{78}$  & $2$--$4^*$ & ${}^{2}a^{9}_{79}$  & $2$--$4^*$ & ${}^{2}a^{9}_{80}$  & $2$--$4^*$ & ${}^{2}a^{9}_{81}$  & $2$--$4^*$ \\
 ${}^{2}a^{9}_{82}$  & $2$--$4^*$ & ${}^{2}a^{9}_{83}$  & $2$--$4^*$ & ${}^{3}a^{9}_{84}$  & $3$ & ${}^{3}a^{9}_{85}$  & $3$ & ${}^{3}a^{9}_{86}$  & $3$ \\
 ${}^{3}a^{9}_{88}$  & $4$ & ${}^{3}a^{9}_{89}$  & $4$ & ${}^{3}a^{9}_{90}$  & $4$ & ${}^{3}a^{9}_{91}$  & $4$ & ${}^{3}a^{9}_{92}$  & $4$ \\
 ${}^{3}a^{9}_{93}$  & $4$ & ${}^{3}a^{9}_{94}$  & $2$--$4^*$ & ${}^{3}a^{9}_{95}$  & $4$ & ${}^{4}a^{9}_{96}$  & $4$ & ${}^{2}n^{9}_{105}$  & $2$ \\
 ${}^{2}n^{9}_{106}$  & $2^*$ & ${}^{2}n^{9}_{107}$  & $2$ & ${}^{2}n^{9}_{108}$  & $2$ & ${}^{2}n^{9}_{109}$  & $2^*$ & ${}^{2}n^{9}_{110}$  & $2^*$ \\
 ${}^{2}n^{9}_{111}$  & $2^*$ & ${}^{2}n^{9}_{112}$  & $2^*$ & ${}^{2}n^{9}_{113}$  & $2$ & ${}^{2}n^{9}_{114}$  & $2^*$ & ${}^{2}n^{9}_{115}$  & $2^*$ \\
 ${}^{2}n^{9}_{116}$  & $2^*$ & ${}^{2}n^{9}_{119}$  & $3$ & ${}^{2}n^{9}_{120}$  & $3$ & ${}^{2}n^{9}_{122}$  & $4^*$ & ${}^{2}n^{9}_{123}$  & $4$ \\
 ${}^{3}n^{9}_{124}$  & $3$ & ${}^{3}n^{9}_{125}$  & $3$ & ${}^{3}n^{9}_{126}$  & $3$ & ${}^{3}n^{9}_{127}$  & $4$ & ${}^{3}n^{9}_{128}$  & $4$ \\
 ${}^{3}n^{9}_{129}$  & $2^*$ & ${}^{3}n^{9}_{130}$  & $4$ & ${}^{3}n^{9}_{131}$  & $4$ & ${}^{3}n^{9}_{132}$  & $4$ & ${}^{2}a^{10}_{124}$  & $2^*$ \\
 ${}^{2}a^{10}_{125}$  & $2^*$ & ${}^{2}a^{10}_{126}$  & $2^*$ & ${}^{2}a^{10}_{127}$  & $2^*$ & ${}^{2}a^{10}_{128}$  & $2^*$ & ${}^{2}a^{10}_{129}$  & $2^*$ \\
 ${}^{2}a^{10}_{130}$  & $2^*$ & ${}^{2}a^{10}_{131}$  & $2^*$ & ${}^{2}a^{10}_{132}$  & $2^*$ & ${}^{2}a^{10}_{133}$  & $2^*$ & ${}^{2}a^{10}_{134}$  & $2^*$ \\
 ${}^{2}a^{10}_{135}$  & $2^*$ & ${}^{2}a^{10}_{136}$  & $2^*$ & ${}^{2}a^{10}_{137}$  & $2^*$ & ${}^{2}a^{10}_{138}$  & $2^*$ & ${}^{2}a^{10}_{139}$  & $2^*$ \\
 ${}^{2}a^{10}_{140}$  & $2^*$ & ${}^{2}a^{10}_{141}$  & $2^*$ & ${}^{2}a^{10}_{142}$  & $2^*$ & ${}^{2}a^{10}_{143}$  & $2^*$ & ${}^{2}a^{10}_{144}$  & $2^*$ \\
 ${}^{2}a^{10}_{145}$  & $2^*$ & ${}^{2}a^{10}_{146}$  & $2^*$ & ${}^{2}a^{10}_{147}$  & $2^*$ & ${}^{2}a^{10}_{148}$  & $2^*$ & ${}^{2}a^{10}_{149}$  & $2^*$ \\
 ${}^{2}a^{10}_{150}$  & $2^*$ & ${}^{2}a^{10}_{151}$  & $2^*$ & ${}^{2}a^{10}_{152}$  & $2^*$ & ${}^{2}a^{10}_{153}$  & $2^*$ & ${}^{2}a^{10}_{154}$  & $2^*$ \\
 ${}^{2}a^{10}_{155}$  & $2^*$ & ${}^{2}a^{10}_{156}$  & $2^*$ & ${}^{2}a^{10}_{157}$  & $2^*$ & ${}^{2}a^{10}_{158}$  & $2^*$ & ${}^{2}a^{10}_{159}$  & $2^*$ \\
 ${}^{2}a^{10}_{160}$  & $2^*$ & ${}^{2}a^{10}_{161}$  & $2^*$ & ${}^{2}a^{10}_{162}$  & $2^*$ & ${}^{2}a^{10}_{163}$  & $2^*$ & ${}^{2}a^{10}_{164}$  & $2^*$ \\
 ${}^{2}a^{10}_{165}$  & $2^*$ & ${}^{2}a^{10}_{166}$  & $2^*$ & ${}^{2}a^{10}_{167}$  & $2^*$ & ${}^{2}a^{10}_{168}$  & $2^*$ & ${}^{2}a^{10}_{169}$  & $2^*$ \\
 ${}^{2}a^{10}_{170}$  & $2^*$ & ${}^{2}a^{10}_{171}$  & $2^*$ & ${}^{2}a^{10}_{172}$  & $2^*$ & ${}^{2}a^{10}_{173}$  & $2^*$ & ${}^{2}a^{10}_{181}$  & $3$ \\
 ${}^{2}a^{10}_{190}$  & $3^*$ & ${}^{2}a^{10}_{195}$  & $3$ & ${}^{2}a^{10}_{196}$  & $3$ & ${}^{2}a^{10}_{197}$  & $3^*$ & ${}^{2}a^{10}_{198}$  & $3^*$ \\
 ${}^{2}a^{10}_{200}$  & $3$ & ${}^{2}a^{10}_{201}$  & $3^*$ & ${}^{2}a^{10}_{204}$  & $3$ & ${}^{2}a^{10}_{206}$  & $3$ & ${}^{2}a^{10}_{208}$  & $3$ \\
 ${}^{2}a^{10}_{210}$  & $3^*$ & ${}^{2}a^{10}_{211}$  & $2$--$4^*$ & ${}^{2}a^{10}_{212}$  & $2$--$4^*$ & ${}^{2}a^{10}_{213}$  & $2$--$4^*$ & ${}^{2}a^{10}_{214}$  & $2$--$4^*$ \\
 ${}^{2}a^{10}_{215}$  & $2$--$4^*$ & ${}^{2}a^{10}_{216}$  & $2$--$4^*$ & ${}^{2}a^{10}_{217}$  & $4$ & ${}^{2}a^{10}_{218}$  & $2$--$4^*$ & ${}^{2}a^{10}_{219}$  & $4$ \\
 ${}^{2}a^{10}_{220}$  & $4$ & ${}^{2}a^{10}_{221}$  & $4$ & ${}^{2}a^{10}_{222}$  & $2$--$4^*$ & ${}^{2}a^{10}_{223}$  & $4$ & ${}^{2}a^{10}_{224}$  & $4$ \\
 ${}^{2}a^{10}_{225}$  & $4$ & ${}^{2}a^{10}_{226}$  & $2$--$4^*$ & ${}^{2}a^{10}_{227}$  & $2$--$4^*$ & ${}^{2}a^{10}_{228}$  & $4$ & ${}^{2}a^{10}_{229}$  & $2$--$4^*$ \\
 ${}^{2}a^{10}_{230}$  & $4$ & ${}^{2}a^{10}_{231}$  & $2^*$ & ${}^{2}a^{10}_{232}$  & $2^*$ & ${}^{2}a^{10}_{233}$  & $2^*$ & ${}^{2}a^{10}_{234}$  & $2$--$4^*$ \\
 ${}^{2}a^{10}_{235}$  & $2$--$4^*$ & ${}^{2}a^{10}_{236}$  & $2$--$4^*$ & ${}^{2}a^{10}_{237}$  & $5$ & ${}^{2}a^{10}_{238}$  & $5$ & ${}^{2}a^{10}_{239}$  & $5$ \\
 ${}^{2}a^{10}_{240}$  & $5$ & ${}^{2}a^{10}_{241}$  & $5$ & ${}^{2}a^{10}_{242}$  & $5$ & ${}^{2}a^{10}_{243}$  & $5$ & ${}^{2}a^{10}_{244}$  & $5$ \\
 ${}^{3}a^{10}_{245}$  & $3$ & ${}^{3}a^{10}_{246}$  & $3$ & ${}^{3}a^{10}_{247}$  & $3$ & ${}^{3}a^{10}_{248}$  & $3$ & ${}^{3}a^{10}_{249}$  & $3$ \\
 ${}^{3}a^{10}_{251}$  & $4$ & ${}^{3}a^{10}_{252}$  & $4$ & ${}^{3}a^{10}_{253}$  & $4$ & ${}^{3}a^{10}_{254}$  & $4$ & ${}^{3}a^{10}_{255}$  & $4$ \\
 ${}^{3}a^{10}_{256}$  & $4$ & ${}^{3}a^{10}_{257}$  & $4$ & ${}^{3}a^{10}_{258}$  & $4$ & ${}^{3}a^{10}_{259}$  & $2$--$4^*$ & ${}^{3}a^{10}_{260}$  & $2$--$4^*$ \\
 ${}^{3}a^{10}_{261}$  & $2$--$4^*$ & ${}^{3}a^{10}_{262}$  & $4$ & ${}^{3}a^{10}_{263}$  & $2$--$4^*$ & ${}^{3}a^{10}_{264}$  & $2$--$4^*$ & ${}^{3}a^{10}_{265}$  & $5$ \\
 ${}^{3}a^{10}_{266}$  & $5$ & ${}^{3}a^{10}_{267}$  & $5$ & ${}^{3}a^{10}_{268}$  & $5$ & ${}^{3}a^{10}_{269}$  & $5$ & ${}^{3}a^{10}_{270}$  & $3$--$5$ \\
 ${}^{3}a^{10}_{271}$  & $3$--$5$ & ${}^{3}a^{10}_{272}$  & $4$ & ${}^{3}a^{10}_{273}$  & $4$ & ${}^{3}a^{10}_{274}$  & $4$ & ${}^{3}a^{10}_{275}$  & $5$ \\
 ${}^{3}a^{10}_{276}$  & $5$ & ${}^{3}a^{10}_{277}$  & $4$ & ${}^{3}a^{10}_{278}$  & $4$ & ${}^{3}a^{10}_{279}$  & $2$--$4^*$ & ${}^{3}a^{10}_{280}$  & $2$--$4^*$ \\
 ${}^{3}a^{10}_{281}$  & $2$--$4^*$ & ${}^{3}a^{10}_{282}$  & $5$ & ${}^{3}a^{10}_{283}$  & $5$ & ${}^{3}a^{10}_{284}$  & $5$ & ${}^{3}a^{10}_{285}$  & $3$--$5$ \\
 ${}^{3}a^{10}_{287}$  & $5$ & ${}^{4}a^{10}_{288}$  & $4$ & ${}^{4}a^{10}_{289}$  & $5$ & ${}^{4}a^{10}_{290}$  & $5$ & ${}^{4}a^{10}_{291}$  & $5$ \\
 ${}^{4}a^{10}_{293}$  & $2$--$4^*$ & ${}^{4}a^{10}_{294}$  & $5$ & ${}^{4}a^{10}_{295}$  & $5$ & ${}^{4}a^{10}_{296}$  & $5$ & ${}^{5}a^{10}_{297}$  & $5$ \\
 ${}^{2}n^{10}_{340}$  & $2$ & ${}^{2}n^{10}_{341}$  & $2^*$ & ${}^{2}n^{10}_{342}$  & $2$ & ${}^{2}n^{10}_{343}$  & $2^*$ & ${}^{2}n^{10}_{344}$  & $2$ \\
 ${}^{2}n^{10}_{345}$  & $2^*$ & ${}^{2}n^{10}_{346}$  & $2^*$ & ${}^{2}n^{10}_{347}$  & $2$ & ${}^{2}n^{10}_{348}$  & $2$ & ${}^{2}n^{10}_{349}$  & $2$ \\
 ${}^{2}n^{10}_{350}$  & $2^*$ & ${}^{2}n^{10}_{351}$  & $2^*$ & ${}^{2}n^{10}_{352}$  & $2$ & ${}^{2}n^{10}_{353}$  & $2^*$ & ${}^{2}n^{10}_{354}$  & $2^*$ \\
 ${}^{2}n^{10}_{355}$  & $2^*$ & ${}^{2}n^{10}_{356}$  & $2^*$ & ${}^{2}n^{10}_{357}$  & $2$ & ${}^{2}n^{10}_{358}$  & $2^*$ & ${}^{2}n^{10}_{359}$  & $2^*$ \\
 ${}^{2}n^{10}_{360}$  & $2^*$ & ${}^{2}n^{10}_{361}$  & $2^*$ & ${}^{2}n^{10}_{362}$  & $2^*$ & ${}^{2}n^{10}_{363}$  & $2^*$ & ${}^{2}n^{10}_{364}$  & $2$ \\
 ${}^{2}n^{10}_{365}$  & $2^*$ & ${}^{2}n^{10}_{366}$  & $2^*$ & ${}^{2}n^{10}_{367}$  & $2^*$ & ${}^{2}n^{10}_{368}$  & $2^*$ & ${}^{2}n^{10}_{369}$  & $2^*$ \\
 ${}^{2}n^{10}_{370}$  & $2^*$ & ${}^{2}n^{10}_{371}$  & $2$ & ${}^{2}n^{10}_{372}$  & $2^*$ & ${}^{2}n^{10}_{373}$  & $2^*$ & ${}^{2}n^{10}_{374}$  & $2^*$ \\
 ${}^{2}n^{10}_{375}$  & $2$ & ${}^{2}n^{10}_{376}$  & $2^*$ & ${}^{2}n^{10}_{377}$  & $2^*$ & ${}^{2}n^{10}_{378}$  & $2^*$ & ${}^{2}n^{10}_{385}$  & $3$ \\
 ${}^{2}n^{10}_{386}$  & $3$ & ${}^{2}n^{10}_{388}$  & $3^*$ & ${}^{2}n^{10}_{390}$  & $3$ & ${}^{2}n^{10}_{392}$  & $3$ & ${}^{2}n^{10}_{393}$  & $2$--$4^*$ \\
 ${}^{2}n^{10}_{394}$  & $4$ & ${}^{2}n^{10}_{395}$  & $2$--$4^*$ & ${}^{2}n^{10}_{396}$  & $2$--$4^*$ & ${}^{2}n^{10}_{397}$  & $2^*$ & ${}^{2}n^{10}_{398}$  & $2^*$ \\
 ${}^{2}n^{10}_{399}$  & $2^*$ & ${}^{2}n^{10}_{400}$  & $2^*$ & ${}^{2}n^{10}_{401}$  & $2^*$ & ${}^{2}n^{10}_{402}$  & $2$--$4^*$ & ${}^{2}n^{10}_{403}$  & $2$--$4^*$ \\
 ${}^{3}n^{10}_{404}$  & $3$ & ${}^{3}n^{10}_{405}$  & $3$ & ${}^{3}n^{10}_{406}$  & $3$ & ${}^{3}n^{10}_{407}$  & $4$ & ${}^{3}n^{10}_{408}$  & $4$ \\
 ${}^{3}n^{10}_{409}$  & $2^*$ & ${}^{3}n^{10}_{410}$  & $4$ & ${}^{3}n^{10}_{411}$  & $4$ & ${}^{3}n^{10}_{412}$  & $2$--$4^*$ & ${}^{3}n^{10}_{413}$  & $5$ \\
 ${}^{3}n^{10}_{414}$  & $5$ & ${}^{3}n^{10}_{415}$  & $3^*$ & ${}^{3}n^{10}_{416}$  & $5$ & ${}^{3}n^{10}_{417}$  & $5$ & ${}^{3}n^{10}_{418}$  & $3^*$ \\
 ${}^{3}n^{10}_{419}$  & $4$ & ${}^{3}n^{10}_{420}$  & $5$ & ${}^{3}n^{10}_{421}$  & $5$ & ${}^{3}n^{10}_{422}$  & $4$ & ${}^{3}n^{10}_{423}$  & $5$ \\
 ${}^{3}n^{10}_{424}$  & $3$--$5$ & ${}^{3}n^{10}_{425}$  & $3$--$5$ & ${}^{3}n^{10}_{426}$  & $5$ & ${}^{3}n^{10}_{427}$  & $4$ & ${}^{3}n^{10}_{428}$  & $4$ \\
 ${}^{3}n^{10}_{429}$  & $4$ & ${}^{3}n^{10}_{430}$  & $4$ & ${}^{3}n^{10}_{431}$  & $5$ & ${}^{3}n^{10}_{432}$  & $5$ & ${}^{3}n^{10}_{433}$  & $5$ \\
 ${}^{3}n^{10}_{434}$  & $5$ & ${}^{4}n^{10}_{435}$  & $5$ & ${}^{4}n^{10}_{436}$  & $5$ & ${}^{4}n^{10}_{437}$  & $5$ & ${}^{4}n^{10}_{438}$  & $5$ \\
 ${}^{4}n^{10}_{439}$  & $3$--$5$ & ${}^{4}n^{10}_{440}$  & $5$ & ${}^{4}n^{10}_{441}$  & $5$ & ${}^{4}n^{10}_{442}$  & $3$--$5$ & ${}^{4}n^{10}_{443}$  & $5$ \\
 ${}^{4}n^{10}_{444}$  & $5$ & ${}^{4}n^{10}_{445}$  & $2$--$4^*$ & ${}^{4}n^{10}_{446}$  & $2^*$ & ${}^{4}n^{10}_{447}$  & $5$ & ${}^{4}n^{10}_{448}$  & $5$ \\
 ${}^{4}n^{10}_{449}$  & $5$ & ${}^{4}n^{10}_{450}$  & $5$ & ${}^{5}n^{10}_{451}$  & $5$ & ${}^{5}n^{10}_{452}$  & $5$ \\
\bottomrule
\end{longtable}
\end{center}

\end{document}